\documentclass{amsart}

\usepackage[sort&compress, round]{natbib}

\usepackage[normalem]{ulem}

\usepackage{amsmath, amssymb, amsthm, amscd, verbatim}

\usepackage{enumitem}

\usepackage[usenames,dvipsnames]{color}

\usepackage{tikz-cd}

\usepackage{mathabx}

\usepackage{hyperref}
\hypersetup{
           breaklinks=true,   
           colorlinks=true,   
			citecolor = Green
        }

\newcommand{\R}  {{\mathbb R}}
\newcommand{\C}  {{\mathbb C}}
\newcommand{\K}  {{\mathbb K}}
\newcommand{\N}  {{\mathbb N}}
\newcommand{\Z}  {{\mathbb Z}}
\newcommand{\ba} {{\boldsymbol{\alpha}}}
\newcommand{\bb} {{\boldsymbol{\beta}}}
\newcommand{\bg} {{\boldsymbol{\gamma}}}
\newcommand{\bsa} {{\boldsymbol{a}}}
\newcommand{\bm} {{\boldsymbol{m}}}
\newcommand{\bn} {{\boldsymbol{\nu}}}
\newcommand{\bx} {{\boldsymbol{x}}}
\newcommand{\by} {{\boldsymbol{y}}}
\newcommand{\bz} {{\boldsymbol{z}}}
\newcommand{\bu} {{\boldsymbol{u}}}
\newcommand{\bv} {{\boldsymbol{v}}}
\newcommand{\bof}{{\boldsymbol{f}}}
\newcommand{\bog}{{\boldsymbol{g}}}
\newcommand{\bN} {{\boldsymbol{N}}}
\newcommand{\bU} {{\boldsymbol{U}}}
\newcommand{\X}  {{\mathfrak X}}
\newcommand{\Y}  {{\mathfrak Y}}
\newcommand{\KK} {{\mathcal{K}}}

\newcommand{\eps}  {\varepsilon}
\renewcommand{\phi}{\varphi}

\newcommand{\scp}[2]{\langle #1, #2 \rangle}
\newcommand{\spann} {\operatorname{span}}

\newcommand{\ha}{A1}
\newcommand{\hb}{A2}
\newcommand{\hc}{A3}

\newcommand{\lo}{\downarrow}
\newcommand{\up}{\uparrow}

\newcommand{\Act}{\operatorname{Act}}

\DeclareMathOperator{\decay}{decay}
\DeclareMathOperator{\dec}{dec}
\DeclareMathOperator{\err}{err}
\DeclareMathOperator{\cost}{cost}

\newtheorem{lemma}{Lemma}[section]
\newtheorem{theo}[lemma]{Theorem}
\newtheorem{corollary}[lemma]{Corollary}

\theoremstyle{definition}
\newtheorem{rem}[lemma]{Remark}

\usepackage[usenames,dvipsnames]{color}

\begin{document}

\title[]{Infinite-dimensional integration and $L^2$-approximation on 
Hermite spaces}

\author[Gnewuch]
{M.~Gnewuch}
\address{
Institut f\"ur Mathematik\\
Universit\"at Osnabr\"uck\\
Albrechtstra{\ss}e \ 28A\\
49076 Osnabr\"uck\\ 
Germany}
\email{michael.gnewuch@uni-osnabrueck.de}

\author[Hinrichs]
{A.~Hinrichs}
\address{
Institut f\"ur Analysis\\
Johannes-Kepler-Universit\"at Linz\\
Altenberger Str.\ 69\\
4040 Linz\\
Austria}
\email{aicke.hinrichs@jku.at}

\author[Ritter]
{K.~Ritter}
\address{Fachbereich Mathematik\\
Rheinland-Pfälzische Technische Universit\"at Kaisers\-lautern-Landau\\
Postfach 3049\\
67653 Kaiserslautern\\
Germany}
\email{ritter@mathematik.uni-kl.de}

\author[R\"u\ss mann]
{R.~R\"u\ss mann}
\address{Fachbereich Mathematik\\
Rheinland-Pfälzische Technische Universit\"at Kaisers\-lautern-Landau\\
Postfach 3049\\
67653 Kaiserslautern\\
Germany}
\email{ruessmann@mathematik.uni-kl.de}

\date{November 04, 2023}


\begin{abstract}
We study integration and $L^2$-approximation of functions 
of infinitely many variables in the following setting:
The underlying function space 
is the countably infinite tensor product of univariate Hermite spaces 
and the probability measure is the corresponding product
of the standard normal distribution.
The maximal domain of the functions from this tensor product space
is necessarily a proper subset of the sequence space $\R^\N$.
We establish upper and lower bounds for the minimal
worst case errors under general assumptions;
these bounds do match for tensor products of
well-studied Hermite spaces of functions with finite or 
with infinite smoothness.
In the proofs we employ embedding results,
and the upper bounds are attained constructively with the help of
multivariate decomposition methods.
\end{abstract}

\maketitle

\section{Introduction}\label{s1}

We study integration and $L^2$-approximation, based on function 
evaluations, in the worst-case setting on the unit ball in a
reproducing kernel Hilbert space (RKHS) $H(K)$. In this paper 
the space $H(K)$ consists of functions of infinitely many real 
variables, and it is obtained as the tensor product of
RKHSs $H(k_j)$ of functions of a single variable. More precisely, 
the measure $\mu$ on $\R^\N$ that defines the integral and
the $L^2$-norm is the infinite product of the standard normal
distribution $\mu_0$ on $\R$, and the reproducing kernel 
\[
K \colon \X \times \X \to \R
\]
is the infinite tensor product of Hermite kernels 
$k_j \colon \R \times \R \to \R$
on a suitable domain $\X$ in the sequence space $\R^\N$.

In general, a univariate \emph{Hermite kernel}
$k_j$ is defined in terms of the
Hermite polynomials $h_\nu$ of degree $\nu \in \N_0$, normalized in the 
space $L^2(\mu_0)$, and of \emph{Fourier weights} $\alpha_{\nu,j}$ 
with $\nu \in \N$ in the following way: Assuming
\[
\inf_{\nu \in \N} \alpha_{\nu,j} > 0 \qquad \text{and} \qquad
\sum_{\nu\in \N} \alpha_{\nu,j}^{-1} \cdot \nu^{-1/2} < \infty,
\]
we define
\begin{equation}\label{g22}
\phantom{\qquad\quad x,y \in \R.}
k_j (x,y) := 
1 +\sum_{\nu\in \N} 
\alpha_{\nu,j}^{-1} \cdot h_\nu (x) \cdot h_\nu (y),
\qquad\quad x,y \in \R.
\end{equation}
The regularity of the functions from the \emph{Hermite space}
$H(k_j) \subseteq L^2(\mu_0)$ is determined by the asymptotic behavior 
of $\alpha_{\nu,j}$ as $\nu \to \infty$.
See, e.g., \citet[Sec.~3.1]{GHHR2021} and
\citet[Sec.~2]{LPE22} for details and further references.

For finite tensor products 
$H(k_1) \otimes \dots \otimes H(k_d)$ integration and 
$L^2$-approximation with respect to the
$d$-dimensional standard normal distribution
$\mu_0 \otimes \dots \otimes \mu_0$ have 
been studied by 
\citet{IL15}, who have distinguished between
two asymptotic regimes for the Fourier weights, namely
one of polynomial and one of (sub-)exponential growth 
as $\nu \to \infty$.
The former case is further addressed in
\citet{DILP18}, \citet{KSG22}, \citet{DungNguyen22}, 
and \citet{LPE22}, while the
latter is further studied in
\citet{IKLP15}, \citet{IKPW16a}, and \citet{IKPW16b}.

In the present paper we consider the infinite tensor product
\[
H(K) \simeq \bigotimes_{j \in \N} H(k_j)
\]
as well as the infinite product $\mu := \bigotimes_{j \in \N} \mu_0$ 
on the space $\R^\N$. The tensor product kernel $K$,
given by
\[
\phantom{\qquad\quad \bx, \by \in \X,}
K(\bx,\by) := \prod_{j \in \N} k_j(x_j,y_j),
\qquad\quad \bx, \by \in \X,
\]
on the maximal domain
\begin{equation}\label{g23}
\X :=  \Bigl\{ \bx \in \R^\N \colon 
\sum_{\nu,j \in \N} \alpha_{\nu,j}^{-1} \cdot h^2_\nu (x_j) < \infty 
\Bigr\},
\end{equation}
will be called a \emph{Hermite kernel}, too, and $H(K)$,
which is a space of real-valued functions with domain $\X$, will
be called a \emph{Hermite space}, too.
In order to
analyze both of the asymptotic regimes for the Fourier weights in
parallel, we introduce the summability and monotonicity assumptions 
\begin{itemize}
\label{a123}
\item[(\ha)]
$0 < \alpha_{1,j} \leq \alpha_{2,j} \leq \dots$
for every $j \in \N$,
\item[(\hb)]
$\sum_{\nu,j \in \N} \alpha_{\nu,j}^{-1} < \infty$,
\item[(\hc)]
$\sum_{j \in \N} \gamma_j < \infty$, 
where $\gamma_j := \sup_{\nu \in \N} \alpha_{\nu,1}/\alpha_{\nu,j}$.
\end{itemize}

We employ the unrestricted subspace 
sampling model, which has been introduced in
\citet{KuoEtAl10}, to
define the cost of algorithms for computational problems on spaces
of functions of infinitely many variables. 
The key quantity in our analysis is
the $n$-th minimal error $\err_n(K)$
for integration or $L^2$-approximation, which is, roughly speaking,
the smallest worst-case error on the unit ball in $H(K)$
that can be achieved by any deterministic algorithm with cost at most 
$n \in \N$. 

For a sequence $\bz$ of positive real numbers we use 
$\decay(\bz) \in [0,\infty]$ to denote its polynomial order of 
convergence towards zero, 
see \eqref{g90} for the precise definition.
Our two main results deal with the decay of the $n$-th
minimal errors $\dec(K) := \decay (\err_n(K)_{n \in \N})$.
Let $\ba_1^{-1} := (\alpha_{1,j}^{-1})_{j\in \N}$ and 
$\bg := (\gamma_j)_{j \in \N}$, and let $\dec(k_1)$ denote
the decay of the $n$-th minimal errors for the 
univariate integration or $L^2$-approximation problem on the unit ball 
in $H(k_1)$. Under the assumptions (\ha)--(\hc) we obtain 
\begin{equation}\label{g24}
\min \left( \dec (k_1), \frac{\decay(\bg) - 1}{2} \right)
\le \dec (K) \le 
\min \left( \dec (k_1), \frac{\decay(\ba_1^{-1}) - 1}{2} \right)
\end{equation}
for integration and for $L^2$-approximation, see Theorem~\ref{t3}.
Moreover, we show that for Fourier weights
with a polynomial and with a (sub-)exponential growth
the upper and lower bound in \eqref{g24} match, see Corollary~\ref{t1}.

Let us put the function space setting and the results from the present 
paper into a general perspective by considering an arbitrary probability
measure $\mu_0$ on a domain $D$ as well as any orthonormal basis
$(h_\nu)_{\nu \in \N_0}$ of $L^2(\mu_0)$ with $h_0 = 1$.
Under suitable summability and monotonicity assumptions on the Fourier 
weights, the definition \eqref{g22}, with $D$ in place of $\R$,
yields a sequence of reproducing kernels $k_j$ on the domain $D$
with $H(k_j) \subseteq L^2(\mu_0)$.
Furthermore, the maximal domain $\X$ of the corresponding tensor product
kernel $K$ is given by \eqref{g23}, again with $D$ in place of $\R$.

In this general setting, integration and $L^2$-approximation with respect
to $\mu := \bigotimes_{j \in \N} \mu_0$ are studied in \citet{GHHRW2020}, 
and \eqref{g24} is established under the assumption that 
\begin{equation}\label{g25}
\X = D^\N,
\end{equation}
see \citet[Thm.~4.3]{GHHRW2020}. Moreover, matching bounds for
$\dec(K)$ are derived from \eqref{g24} for various kinds of spaces 
$H(k_j)$ on compact domains $D$, e.g., for
Korobov, Walsh, Haar, and Sobolev spaces. 

In the general setting and given \eqref{g25},
the method of proof for the lower bound in \eqref{g24}, i.e.,
for an upper bound for $\err_n(K)$, is to embed $H(K)$ into
a particular RKHS $H(M^\up_a)$ 
whose reproducing kernel is 
a tensor product involving weighted anchored kernels.
More precisely,
for a fixed $a \in D$ and with a suitable reproducing kernel $m_a
\colon D \times D \to \R$ such that 
\begin{equation}\label{g86}
m_a^{\up}(a,a) = 0
\end{equation}
we consider
\begin{equation}\label{g87}
\phantom{\qquad\quad \bx, \by \in D^\N.}
M^\up_a(\bx,\by) := \prod_{j \in \N} (1 + 
\gamma_j \cdot m_a^\up(x_j,y_j)),
\qquad\quad \bx, \by \in D^\N.
\end{equation}
The structure of $M^\up_a$ is most 
favorable for the construction and analysis of deterministic algorithms,
since the multilevel method and the multivariate decomposition method 
(MDM) are available as powerful meta-algorithms for this type 
of RKHS. 
Multilevel methods on spaces $H(M_a^\up)$ with \eqref{g86} and \eqref{g87}
have first been considered in \citet{NHMR2011}; see, e.g.,
\citet{Gne10}, 
\citet{Gne12a},  
\citet{DG12}, 
\citet{GneEtAl16}
for subsequent work. 
MDMs on spaces $H(M_a^\up)$ with \eqref{g86} and \eqref{g87}
have first been considered in \citet{KuoEtAl10};
see, e.g.,
\citet{MR2805529},
\citet{PW11},
\citet{Was12},
\citet{Gne12a},
\citet{DG12},
\citet{Was13},
\citet{KuoEtAl2017}
for subsequent work. 
For multivariate integration 
the underlying idea of MDMs was used for the first time
in \cite{GH2010}; there the corresponding 
algorithms were baptized dimension-wise quadrature methods.

Since $H(K) \subseteq H(M_a^\up)$ with a continuous identical
embedding, lower bounds for $\dec(M_a^\up)$ yield
lower bounds for $\dec(K)$.
The upper bound in \eqref{g24}, which
corresponds to a lower bound for $\err_n(K)$, is established via an
embedding into some suitable reproducing kernel Hilbert space
$H(M_a^\lo)$, too.
We stress that the scales of spaces $H(k_j)$ and $H(1+\gamma_j
\cdot m^{\up}_a)$ are rather different in the following sense:
While $H(1+\gamma_j \cdot m^{\up}_a)$, as a vector space, does not depend
on $j$, we typically have a compact embedding of $H(k_j)$ into
$H(k_i)$ if $j > i$.

Let us return to the setting of Hermite spaces. 
Since all functions $h_\nu$ with $\nu \in \N$ are unbounded on 
their domain $D=\R$,
condition \eqref{g25} is never satisfied. Consequently,
Hermite spaces $H(K)$ are never spaces of functions on the whole 
domain $\R^\N$. However, we have $\mu(\X)=1$ for the maximal
domain $\X$ of $K$ and $H(K)
\subseteq L^2(\mu)$ in the Hermite case under suitable assumptions 
on the Fourier weights, in particular, if (\ha) and (\hb) are satisfied. 
See \citet[Sec.~3.2]{GHHR2021} for details. 
Hermite spaces fit into the abstract setting developed in
\citet{GriebelOswald2017}; the study of natural domains of 
functions of infinitely many variables has not been the main
focus there.

In the present paper we
adapt the embedding approach from \citet{GHHRW2020} to cope with an
RKHS on a proper subset of $\R^\N$. 
We add that in this new setting the corresponding kernels 
$M^\up_a$ and
$M^\lo_a$ have, in general, domains different from the one of $K$.
Hence the corresponding embeddings 
$H(M^\lo_a) \hookrightarrow H(K) \hookrightarrow H(M^\up_a)$
turn out to be typically not identical embeddings but restrictions.
Indeed, in the proof of
the upper bound for the minimal errors we may end up with 
$\X^\up \subsetneq \X$ for the maximal domain $\X^\up$ of
$M_a^\up$, but still $\mu(\X^\up)=1$, while in the 
proof of the lower bound we may 
end up with a domain that is strictly larger than $\X$, 
but still not the whole space $\R^\N$.
The embedding of $H(K)$ into $H(M_a^\up)$ allows to
apply MDMs also on Hermite spaces and, in this way, to
achieve the lower bound in \eqref{g24}.
We stress that we have explicit upper bounds for the norms
of all embedding and restriction maps relevant in this context.

Consider again a general domains $D$ endowed with an arbitrary 
probability measures $\mu_0$. We want to point out that the approach 
developed in this paper can also be used to treat tensor
product kernels $K$ based on an arbitrary orthonormal basis 
$(h_\nu)_{\nu\in\N_0}$ of $L^2(\mu_0)$, consisting of not 
necessarily bounded functions, as long as $h_0=1$ and
suitable upper bounds for the univariate functions 
$h_\nu (x)$ for every $\nu\in \N$
and $x\in D$, similar to the bounds 
for Hermite polynomials listed in 
\citet[Lemma~3.1 and 3.2]{GHHR2021}, are available.

Furthermore, we want to emphasize that Hermite spaces of functions 
depending on one, several, or infinitely many variables are not 
only interesting in their own right, but are intimately related to 
Hilbert spaces of Gaussian kernels. More precisely, for each Hilbert 
space of Gaussian kernels there exists a corresponding Hermite space 
with Fourier weights of exponential growth and a linear isometric 
isomorphism between both spaces which respects point evaluations. 
For more details see \cite{GHHR2021} and \citet{GHRR2022}. 

This paper is organized as follows.
In Section~\ref{s2} we 
briefly discuss Hermite spaces with a focus on Fourier
weights with a polynomial or a (sub-)exponential growth.
The RKHS 
$H(M^\up_a)$ and its counterpart for the lower bound are introduced 
and related to $H(K)$ as well as $L^2(\mu)$ in Section~\ref{s3}.
Section~\ref{s4} is devoted to the study of integration and 
$L^2$-approximation on Hermite spaces. 
Univariate Hermite spaces of finite smoothness are considered
in Section~\ref{s4.1},
and the results from Sections~\ref{s3} and
\ref{s4.1} are combined in Section~\ref{s4.2} to derive the main
results of this paper on integration and $L^2$-approximation
on Hermite spaces $H(K)$ of functions of infinitely 
many variables, which are stated in Theorem~\ref{t3}
and Corollary~\ref{t1}.
The construction of MDMs on $H(K)$, which yield the
lower bound in \eqref{g24}, is studied in
Section~\ref{s4.3}. We establish some technical
results in an appendix: Incomplete tensor products of Hilbert spaces are 
used in Appendix~\ref{a1} to study restriction operators
on RKHSs with tensor product kernels, and a monotonicity property
of univariate Hermite kernels is established in Appendix~\ref{a2}.

\section{Hermite Spaces}\label{s2}

We consider Fourier weights $\alpha_{\nu,j}$ for $\nu,j \in \N$ with the
properties (\ha)--(\hc), see p.~\pageref{a123}.
The properties (\ha) and (\hb), together with $\alpha_{0,j} := 1$,
ensure that  (H1), (H2), and (H3) in 
\citet{GHHR2021} are satisfied, so that the results from 
\citet[Sec.~2, 3]{GHHR2021} are 
applicable to the Hermite spaces $H(k_j)$ of functions of a single
variable with reproducing kernels
\[
\phantom{\qquad\quad x,y \in \R,}
k_j (x,y) := 1 + \sum_{\nu\in \N} 
\alpha_{\nu,j}^{-1} \cdot h_\nu (x) \cdot h_\nu (y),
\qquad\quad x,y \in \R,
\]
and to the Hermite space $H(K)$ of functions of infinitely many
variables with the tensor product kernel
\[
\phantom{\qquad\quad \bx,\by \in \X,}
K(\bx,\by) := \prod_{j \in \N} k_j(x_j,y_j),
\qquad\quad \bx,\by \in \X,
\]
on the maximal domain 
\[
\X :=  
\Bigl\{ \bx \in \R^\N \colon  \sum_{\nu,j \in \N} 
\alpha_{\nu,j}^{-1} \cdot h^2_\nu (x_j) < \infty \Bigr\}.
\]
In particular, $\ell^\infty(\N) \subsetneq \X \subsetneq \R^\N$, 
where $\ell^\infty(\N)$ denotes the space of all bounded
sequences of real numbers,
but $\mu(\X)=1$ and $H(K) \subseteq L^2(\mu)$. See 
\citet[Lem.~3.8, Prop.~3.10, and Lem.~3.12]{GHHR2021}.

Concerning (\hc), we note that 
$H(k_j)$ is the orthogonal sum of the space $H(1)$ of
constant functions and the space
\[
H(k_j-1) = 
\Bigl\{ f \in H(k_j) \colon \int_{\R} f \, d \mu_0 = 0 \Bigr\}.
\]
It follows that $\gamma_j^{1/2}$ is the norm of the identical
embedding of $H(k_j-1)$ into the space $H(k_1-1)$.

\subsection{Two Particular Cases}\label{s2.1}

In the sequel we discuss two particular kinds of Fourier weights:

\begin{itemize}
\item[(PG)]
Let $r_j > 1$ for $j \in \N$ such that
\begin{equation}\label{g1}
\sum_{j \in \N} 2^{-r_j} < \infty
\end{equation}
and
\begin{equation}\label{g1a}
r_1 = \inf_{j \in \N} r_j.
\end{equation}
The Fourier weights with a \emph{polynomial growth} are given by 
\[
\alpha_{\nu,j} := (\nu+1)^{r_j}
\]
for $\nu,j \in \N$.
\item[(EG)]
Let $r_j,b_j > 0$ for $j \in \N$ such that \eqref{g1} 
and \eqref{g1a} are
satisfied and
\begin{equation}\label{g1b}
b_1 = \inf_{j \in \N} b_j. 
\end{equation}
The Fourier weights with a \emph{(sub-)exponential growth} are given by 
\[
\alpha_{\nu,j} := 2^{r_j \cdot \nu^{b_j}}
\]
for $\nu,j \in \N$.
\end{itemize}

\begin{lemma}\label{l10}
In both cases, {\rm (PG)} and {\rm (EG)} we have 
{\rm(\ha)}, {\rm(\hb)}, and {\rm(\hc)} with
\[
\gamma_j = 2^{r_1-r_j}
\]
for every $j \in \N$.
\end{lemma}

\begin{proof}
Obviously, (\ha) is satisfied and
$\sum_{\nu \in \N} \alpha_{\nu,j}^{-1} < \infty$ for every $j \in \N$. 
As shown in \citet[Lem.~3.17]{GHHR2021},
there exists an integer $j_0 \in \N$ such that
\[
\sum_{\nu \in \N,\ j \geq j_0} \alpha_{\nu,j}^{-1} < \infty.
\]
Hence we have (\hb).
In the case (PG) we obtain 
\[
\gamma_j = \sup_{\nu \in \N} (\nu+1)^{r_1-r_j} = 2^{r_1-r_j}
\]
for every $j \in \N$ from \eqref{g1a},
so that 
(\hc) follows from \eqref{g1} in this case. In the sequel
we consider the case (EG), where we have
\[
\gamma_j
 = \sup_{\nu \in \N} 2^{r_1 \cdot \nu^{b_1} - r_j \cdot \nu^{b_j}}.
\]
Using \eqref{g1a} and \eqref{g1b} we obtain
\[
2^{r_1-r_j} \leq \gamma_j
\leq
\sup_{\nu \in \N} 2^{(r_1-r_j) \cdot \nu^{b_1} } = 2^{r_1-r_j}
\]
for every $j \in \N$, and
(\hc) follows from \eqref{g1} also in the case (EG).
\end{proof}

\begin{rem}\label{r2}
We refer to \citet[Prop.~3.10, 3.18, and 3.19]{GHHR2021} for results on 
the maximal domain $\X$ in the cases (PG) and (EG); in particular, 
\[
\X =  
\Bigl\{ \bx \in \R^\N \colon  \sum_{j \in \N} 
2^{-r_j} \cdot x_j^2 < \infty \Bigr\}
\]
in the case (EG) with $b_1 \geq 1$.
\end{rem}

\begin{rem}\label{r1}
For $i,j \in \N$
the inclusion $H(k_j) \subsetneq H(k_i)$ with a compact identical
embedding is equivalent to
\begin{equation}\label{g4}
\lim_{\nu \to \infty} \frac{\alpha_{\nu,i}}{\alpha_{\nu,j}} = 0.
\end{equation}
In the case (PG) we have \eqref{g4} if and only if $r_i < r_j$;
observe that the latter holds true for a fixed $i \in \N$
if $j$ is sufficiently large.
In the case (EG) we have \eqref{g4} if and only if $b_i < b_j$ or
$b_i = b_j$ and $r_i < r_j$;
observe that the latter holds true for $i = 1$
if $j$ is sufficiently large. This setting of compact identical embeddings is also known as \emph{increasing smoothness}, see, e.g., \citet[p.~1-2]{GHHRW2020} for further details and references.
\end{rem}

\begin{rem}
Roughly speaking, the case {\rm (PG)} corresponds to spaces
$H(k_j)$ of functions of finite smoothness, 
while the case {\rm (EG)} corresponds to an infinite smoothness. 

More precisely, in the case {\rm (PG)} with the additional assumption 
$r_j\in\N$, the space $H(k_j)$ is, up to equivalence of norms, the Sobolev
space of all continuous functions in $L^2(\mu_0)$ with weak
derivatives of order up to $r_j$ belonging to $L^2(\mu_0)$.
See \citet[Exmp.~3.5]{GHHR2021} and \citet[Sec.~2]{LPE22} 
for further details and references. 

In the case {\rm (EG)}, the elements of $H(k_j)$ are real analytic 
if $b_j \geq 1$, and they belong to the Gevrey class of index 
$(2b_j)^{-1}$ if $0 < b_j < 1$.
See \citet[Exmp.~3.6, Lem.~3.7]{GHHR2021} for further details and 
references.
\end{rem}

\begin{rem}\label{r5}
Let $r_j > 1/2$ for $j \in \N$ such that \eqref{g1} is
satisfied.
The case (PG) is studied 
in \citet{GHHR2021} under these weaker assumptions, which
still guarantee (\ha), but (\hb) 
is equivalent to $r_j > 1$ for every $j \in \N$,
and (\hc) is equivalent to \eqref{g1a}.
The approach from the present paper may be generalized 
in the following way to also work under the weaker assumptions.

The reproducing kernel $k_j$ is well defined already if 
\[
\sum_{\nu \in \N} \alpha_{\nu,j}^{-1} \cdot \nu^{-1/2} < \infty,
\]
see \citet[Sec.~3.1]{GHHR2021},
and the latter is equivalent to $r_j > 1/2$. 
Due to \eqref{g1} we may
choose $j_0 \in \N$ such that $r_j > 1$ for $j \geq j_0$,
and, possibly after reordering the kernels, we have $r_{j_0} =
\inf_{j \geq j_0} r_j$.
Hence we may study
$H(k_1) \otimes \dots \otimes H(k_{j_0-1}) \otimes H(K^\prime)$,
where $K^\prime$ is the tensor product of the reproducing kernels 
$k_j$ with $j \geq j_0$. For the analysis of $H(K^\prime)$ we then
use $H(k_{j_0})$ instead of $H(k_1)$ as the reference space, for 
instance in the definition of $\gamma_j$ for $j \geq j_0$.
\end{rem}

\begin{rem}\label{r6}
Let $r_j,b_j > 0$ for $j \in \N$ such that \eqref{g1} and
$b_0 := \inf_{j \in \N} b_j > 0$ are satisfied.
The case (EG) is studied 
in \citet{GHHR2021} under these weaker assumptions, which
neither affect (\ha) nor (\hb), but (\hc) is equivalent to
$b_0 = b_1$ and $r_1  \leq r_j$ for every $j \in \N$ with $b_1=b_j$.
Let $r_0 := \inf_{j \in \N} r_j$.
The approach from the present paper may be generalized to also work 
under the weaker assumptions
by studying $H(k_0) \otimes H(K)$, where $k_0$ is defined by means
of the Fourier weights 
$\alpha_{\nu,0} := 2^{ r_0 \cdot \nu^{b_0}}$ for $\nu \in \N$ 
and $\alpha_{0,0} := 1$.
In the analysis we then use $H(k_0)$ instead of $H(k_1)$ as the
reference space.
\end{rem}

Finite tensor products of Hermite spaces with different kinds
of Fourier weights have been studied, e.g., in the following
papers on integration and approximation problems.
Fourier weights with a polynomial growth are considered in 
\citet{IL15}, \citet{DILP18}, \citet{KSG22},
\citet{DungNguyen22}, and \citet{LPE22}.
For Fourier weights with a (sub-)exponential growth we refer to
\citet{IL15}, \citet{IKLP15}, \citet{IKPW16a}, and \citet{IKPW16b}.

\section{Embeddings}\label{s3}

In the sequel we consider Fourier weights $\alpha_{\nu,j}$
that satisfy (\ha)--(\hc). We adapt the approach and some results from 
\citet{GneEtAl16,GHHRW2020} to the case of Hermite spaces; 
the main difference to the previous work is caused by the fact
that the maximal domain $\X$ for a Hermite space $H(K)$ of functions of 
infinitely many variables is always a proper subset of the
sequence space $\R^\N$.

\subsection{The Embedding for the Upper Bounds}\label{s3.1}

To establish upper bounds for worst-case errors 
for integration and $L^2$-approximation on the unit ball
in $H(K)$ we will introduce reproducing kernels $M^\up_a$ with more
favorable structural properties than $K$ in Theorem~\ref{t2}. 
It will turn out that
the maximal domain of each of the kernels $M_a^\up$ with $a \in \R$
is equal to
\[
\X^\up := 
\Bigl\{ \bx \in \R^\N \colon 
\sum_{\nu,j \in\N} \gamma_j \alpha_{\nu,1}^{-1} \cdot h_\nu^2(x_j)
< \infty \Bigr\}.
\]

\begin{lemma}\label{l22}
We have $\ell^\infty(\N) \subsetneq \X^\up \subseteq \X$ and 
$\mu(\X^\up)=1$.
\end{lemma}

\begin{proof}
Since 
\begin{equation}\label{g56}
\frac{\alpha_{\nu,1}}{\gamma_j} \leq \alpha_{\nu,j}
\end{equation}
by definition of $\gamma_j$, we obtain $\X^\up \subseteq \X$.
Using (\hb) and (\hc) we obtain
\[
\int_{\R^\N} 
\sum_{\nu,j \in\N} \gamma_j \alpha_{\nu,1}^{-1} \cdot h_\nu^2(x_j)
\, d \mu (\bx)
=
\sum_{j \in\N} \gamma_j \cdot
\sum_{\nu \in\N} \alpha_{\nu,1}^{-1} < \infty, 
\] 
and therefore $\mu(\X^\up) = 1$. 
For every $\bx \in \ell^\infty(\N)$ we use Cram\'er's inequality to obtain
\[
\sup_{\nu,j \in \N} h^2_\nu(x_j) \leq 
\sup_{j \in \N} \exp(x_j^2/2) < \infty,
\]
which implies $\bx \in \X^\up$. Since $\mu(\ell^\infty(\N))=0$,
we have $\ell^\infty(\N) \neq \X^\up$.
\end{proof}

For $a \in \R$ we put
\[
c^\up(a) := 1 + k_1(a,a).
\]

\begin{theo}\label{t2}
For every $a \in \R$ there exists a reproducing kernel 
\[
m_a^\up \colon \R \times \R \to \R
\]
with the following properties: 
\begin{itemize}
\item[(i)]
We have 
\[
H(k_1) = H(1+m_a^\up)
\]
as vector spaces, and 
\[
m_a^\up(a,a) = 0.
\]
The operator norms of the identical embeddings
$T_j \colon H(k_j) \hookrightarrow H(1+ \gamma_j m_a^\up)$ 
and $V \colon H(m_a^\up) \hookrightarrow H(k_1)$ satisfy 
\[
\| T_j \| \le \left( 1 + \gamma_j c^{\up}(a) \right)^{1/2}
\]
for every $j \in \N$ as well as
\[
\| V \| \le c^{\up}(a).
\]
\item[(ii)]
The maximal domain of the reproducing kernel 
\begin{equation}\label{eq:kernel_M_a_up}
M^\up_a (\bx,\by) := \prod_{j \in \N} (1+ \gamma_j \cdot m_a^\up(x_j,y_j))
\end{equation}
is equal to $\X^\up$.
\item[(iii)]
We have
\[
\{ f|_{\X^\up} \colon f \in H(K) \} \subseteq H(M^\up_a) 
\subseteq L^2(\mu),
\]
and the operator norm of the restriction
$T \colon H(K) \to H(M^\up_a)$, $f \mapsto f|_{\X^\up}$ satisfies
\begin{equation}\label{est:bound_op_norm_restriction}
\|T \| \le C^\up (a) 
:= \prod_{j\in \N} (1+ \gamma_j \cdot c^\up(a) )^{1/2}.
\end{equation}
\end{itemize}
\end{theo}

\begin{proof}
Recall that $H(k_1)$ is the orthogonal sum of the spaces $H(1)$ and
$H(k^*)$, where $k^* := k_1 - 1$. Moreover, we have $\gamma_1 = 1$. 
We put 
\begin{equation}\label{g21}
\xi(f) := \int_\R f \, d \mu_0.
\end{equation}

In a first embedding step we consider the reproducing kernels
\[
k^\up_{j} := 1 + \gamma_j k^* 
\]
for $j \in \N$. We have
\[
H (k^\up_{j}) = H(k_1)
\]
as vector spaces, 
and
\[
\|f\|^2_{H(k^\up_j)} = 
|\xi(f)|^2
+ \gamma_j^{-1} \cdot \|f\|^2
\]
for $f \in H(k_1)$, where
\[
\|f\|^2 := \sum_{\nu \in \N} \alpha_{\nu,1} \cdot
\scp{f}{h_\nu}^2_{L^2(\mu_0)}, 
\]
cf.\ \citet[Prop.~2.1]{GHHR2021}.
Using \eqref{g56} we conclude that $H(k_j) \subseteq H(k^\up_j)$
with an identical embedding 
\[
T_j^{(1)} \colon H(k_j) \hookrightarrow H(k^\up_{j})
\]
of norm one. 

Next, we describe a second embedding step, which depends on $a \in \R$.
According to \citet[Lem.~3.13]{GHHRW2020} there
exists a reproducing kernel $\widehat{m}_a$ that satisfies 
$H(k_1) = H(1+\widehat{m}_a)$ and $\widehat{m}_a(a,a)=0$
as well as $H(1) \cap H(\widehat{m}_a) = \{0\}$ and
\[
\|f\|_{H(1+ \gamma \widehat{m}_a)}^2 = 
|f(a)|^2 + \gamma^{-1} \cdot \|f\|^2
\]
for every $f \in H(k_1)$ and every $\gamma > 0$.
Actually, a slightly different set of assumptions on the Fourier weights
has been imposed in \citet{GHHRW2020}, but the 
proof of the lemma is
applicable without any changes also in the present setting.

Let $f \in H(k_1)$. Since
\[
| f(a) - \xi(f)| \leq
k_1(a,a)^{1/2} \cdot \|f - \xi(f)\|_{H(k_1)} = 
k_1(a,a)^{1/2} \cdot \|f\|, 
\]
we obtain
\begin{align*}
\|f\|_{H(1+ \gamma \widehat{m}_a)}^2 
&\leq
\left( |\xi(f)| + |f(a) - \xi(f)| \right)^2
+ \gamma^{-1} \cdot \|f\|^2\\
&\leq
\left( |\xi(f)| + k_1(a,a)^{1/2} \cdot \|f\| \right)^2
+ \gamma^{-1} \cdot \|f\|^2\\
&\leq
(1+\gamma) \cdot |\xi(f)|^2 + \left( 1 + \gamma^{-1} \right)
\cdot k_1(a,a) \cdot \|f\|^2
+ \gamma^{-1} \cdot \|f\|^2\\
&\leq 
(1+\gamma) \cdot
\left( |\xi(f)|^2 + \gamma^{-1}
\cdot k_1(a,a) \cdot \|f\|^2
+ \gamma^{-1} \cdot \|f\|^2 \right)\\
&=
(1+\gamma) \cdot
\left( |\xi(f)|^2 + 
\gamma^{-1} c^\up(a) \cdot \|f\|^2 \right).
\end{align*}
For $f \in H(\widehat{m}_a)$ we also obtain
\[
\|f\|_{H(k_1)}^2 = |f(a) - \xi(f)|^2 + \|f\|^2
\leq c^\up(a) \cdot \|f\|^2.
\]

Put $m_a^\up := c^\up(a) \cdot \widehat{m}_a$. 
Clearly, $H(k_1) = H(1+m_a^\up)$ and $m_a^\up(a,a) = 0$.
Furthermore,
\[
\|f\|_{H(1+ \gamma \cdot m_a^\up)} \leq
(1+ \gamma c^\up(a))^{1/2} \cdot
\|f\|_{H(1+\gamma k^*)}
\]
for every $f \in H(k_1)$ and every $\gamma > 0$, 
cf.\ \citet[Thm.~2.1]{GneEtAl16} and its proof.
Moreover, 
\[
\|f\|_{H(k_1)} \leq c^\up(a) \cdot \|f\|_{H(m_a^\up)}
\]
for every $f \in H(m_a^\up)$.
Consequently, 
\[
\|T_j\| \leq \|T_j^{(1)}\| \cdot
\|T^{(2)}_j\| \leq (1+ \gamma_j c^\up(a))^{1/2}
\]
for the identical embedding 
\[
T^{(2)}_j \colon H(k^\up_j) \hookrightarrow 
H(1+\gamma_j \cdot m_a^\up)
\]
as well as $\|V\| \leq c^\up(a)$. This completes the proof of
(i).

Due to the equivalence of $\|\cdot\|_{H(k_1)}$ and 
$\|\cdot\|_{H(1+m_a^\up)}$, there exists a constant $c >0$ 
such that 
\[
m_a^\up(x,x) \leq c \cdot(1+k^*(x,x))
\]
and
\[ 
k^*(x,x) \leq c \cdot(1+m_a^\up(x,x))
\]
for all $x \in \R$.  
Let $\bx \in \R^\N$. Employing (\hc) we conclude that 
\[
\sum_{j \in \N} \gamma_j m_a^\up(x_j,x_j) < \infty
\qquad \Leftrightarrow \qquad
\bx \in \X^\up.
\]
Therefore $M^\up_a$ as defined in \eqref{eq:kernel_M_a_up}
is a reproducing kernel with maximal domain $\X^\up$,
i.e., we have (ii).

We claim that
\[
\{ f|_{\X^\up} \colon f \in H(K) \} \subseteq H(M_a^\up).
\]
Indeed,
Lemma~\ref{l22} and Theorem~\ref{ta2}, together with (\hc) and applied 
with $u_j := 1$ for every $j \in \N$, 
$\X^{(\bu)} := \X$, and $\Y^{(\bu)} := \X^\up$, yield the claim
as well as \eqref{est:bound_op_norm_restriction}.

Since
\[
\int_{\R} m_a^\up(x,x) \, d \mu_0(x) \leq
c \cdot \int_{\R} k_1(x,x) \, d \mu_0(x) =
c \cdot 
\biggl( 1+ \sum_{\nu \in \N} \alpha_{\nu,1}^{-1}\biggr)
< \infty,
\]
which follows from (\hb), we obtain
\[
\int_{\X^\up} M^\up_a (\bx,\bx) \, d \mu (\bx)
= \prod_{j \in \N} \left(1 + \gamma_j
\int_{\R} m_a^\up(x,x) \, d \mu_0(x) \right) < \infty
\]
using (\hc).  Consequently, $H(M^\up_a) \subseteq L^2(\mu)$. 
\end{proof}

We use $\bU$ to denote the set of all finite subsets of $\N$.

\begin{rem}\label{r26}
For $\bu \in \bU$ let 
\[
\gamma_\bu := \prod_{j \in \bu} \gamma_j
\]
and
\[
\phantom{\qquad\quad \bx,\by \in \X^\up.}
m^\up_{a,\bu}(\bx, \by) := \prod_{j\in \bu} m^\up_a(x_j, y_j),
\qquad\quad \bx,\by \in \X^\up. 
\]
The reproducing kernel $M^\up_a$ according to Theorem~\ref{t2}
is of the form 
\[
M^\up_a = \sum_{\bu \in \bU} \gamma_\bu \cdot m_{a,\bu}^\up, 
\]
i.e., it is the superposition of weighted tensor product kernels with
weights $\gamma_\bu$ of product form. Additionally, 
$m^\up_a(a,a) = 0$, i.e., $m_a^\up$ is \emph{anchored} at $a$.
These structural properties are most convenient for the construction 
and analysis of deterministic algorithms, since a particular
orthogonal decomposition of $f \in H(M^\up_a)$, the so-called
anchored decomposition, can be computed efficiently using function 
values of $f$ only.
See Section~\ref{s4.3} for details and references. 
Let us mention that similarly to the anchored decomposition, the
functions $f\in H(K)$ also admit a canonical orthogonal decomposition 
into functions depending only on 
finitely many variables. However, in this case we do not have
an anchored decomposition but an ANOVA
decomposition instead, see \citet[Remark~3.14]{GHHR2021}.

We stress that we do not rely on an explicit formula
for the reproducing kernel $m_a^\up$ in the sequel. However, the value
of $c^\up(a)$, or at least of an upper bound, will be needed in
the construction of algorithms, 
where an anchor $a$ with a small value of $c^\up(a)$ is
preferable.
According to Lemma~\ref{l-a1}, we have
\[
\min_{a \in \R} c^\up(a) = c^\up(0).
\]
and using \citet[Eqn.~(5.5.5)]{S75} we obtain
\[
c^\up(0)  = 
2 + \sum_{\nu \in \N} \alpha_{2\nu,1}^{-1} \cdot |h_{2 \nu}(0)|^2 
=
2 + \sum_{\nu \in \N} \alpha_{2\nu,1}^{-1} \cdot
\frac{(2\nu)!}{(\nu! \cdot 2^\nu)^2} 
\leq
2 + \sum_{\nu \in \N} \alpha_{2\nu,1}^{-1}.
\]
\end{rem}

\begin{rem}\label{r4}
The first embedding step in the proof of Theorem~\ref{t2} already
leads to a tensor product kernel $K^\up$ of the form
\[
K^\up (\bx,\by) =
\sum_{\bu \in \bU} 
\gamma_\bu \cdot \prod_{j \in \bu} k^*(x_j,y_j)
\]
with maximal domain $\X^\up$. However,
since $h_1(x) = x$ for every $x \in \R$ and $h_2(0) \neq 0$, the 
reproducing kernel $k^*$ is not an anchored kernel, i.e., we have
$k^*(x,x) > 0$ for every $x \in \R$.
\end{rem}

\begin{rem}\label{ex1}
The inclusion $\X^\up \subseteq
\X$ may be strict, so that the maximal domain may shrink when 
switching from $H(K)$ to $H(K^\up)$ or to $H(M_a^\up)$. Indeed,
let
\[
r_j := \frac{3}{\ln(2)} \cdot \ln(j+1)
\]
for $j \in \N$. The Fourier weights given by 
\[
\alpha_{\nu,j} := 2^{r_j \cdot \nu} = (j+1)^{3\nu}
\]
for $\nu,j \in \N$ satisfy (EG). Due to Remark~\ref{r2},
we have
\[
\X =  \Bigl\{ \bx \in \R^\N \colon 
\sum_{j \in \N} (j+1)^{-3} \cdot x_j^2 < \infty 
\Bigr\}.
\]

Let $\bx \in \R^\N$ be given by $x_j := j^{1/2}$
for $j \in \N$. On the one hand, $\bx \in \X$.
On the other hand, $\bx \not \in \X^\up$, since $\gamma_j
\geq (j+1)^{-3}$, and $h_2(x) = 2^{-1/2} (x^2-1)$, and therefore
\begin{align*}
\sum_{\nu,j \in\N} \gamma_j \alpha_{\nu,1}^{-1} \cdot h_\nu^2(x_j)
& \geq
\sum_{j \in \N} \gamma_j \cdot \alpha_{2,1}^{-1} \cdot
h_2^2(x_j)\\
&\geq
2^{-7} \cdot \sum_{j \in \N} (j-1)^2 \cdot (j+1)^{-3} 
= \infty.
\end{align*}
\end{rem}

\begin{rem}
Both of the spaces $H(K)$ and $H(M_a^\up)$ are infinite
tensor products of reproducing kernel Hilbert spaces. For
$H(M_a^\up)$ the components 
$H(1 + \gamma_j \cdot m_a^\up)$ 
coincide as vector spaces and therefore have equivalent norms.
This is in sharp contrast to the original setting, 
where we often have a
compact identical embedding of $H(k_j)$ into $H(k_i)$ if, for fixed 
$i$, $j$ is sufficiently large, see Remark~\ref{r1}.
\end{rem}

\subsection{The Embedding for the Lower Bounds}\label{s3.2}

The anchored kernel $m_a^\up$ 
and its infinite-dimensional counterpart $M^\up_a$ will 
enable us to prove upper error bounds for integration and
$L^2$-approximation on $H(K)$.
To prove lower error bounds, we need another anchored
kernel 
$M^\lo_a$, whose maximal domain will turn out to be 
\[
\X^\lo := \Bigl\{ \bx \in \R^\N \colon
\sum_{j \in \N} \alpha_{1,j}^{-1} \cdot x_j^2 < \infty\Bigr\}.
\]
Since $h_1(x) = x$ for all $x\in \R$, we have 
$\X \subseteq \X^\lo \subsetneq \R^{\N}$, so that
$\mu(\X^\lo)=1$ follows from $\mu(\X)=1$.

For $a\in \R$ we put
\[
c^\lo(a) := (1 + \alpha_{1,1} + a^2)^{-1}.
\]

\begin{theo}\label{Lemma:Tensorprodukt2_dim_Raeume}
For every $a\in \R$ there exists a reproducing kernel 
$$
m^\lo_a: \R\times \R \to \R
$$
with the following properties:
\begin{itemize}
\item[(i)] We have 
$$
H(1 + m^\lo_a) = \spann\{ h_0, h_1\}
$$ 
as vector spaces and 
$$
m^\lo_a(a,a) = 0.
$$
\item[(ii)] 
The maximal domain of the reproducing kernel 
\[
M^\lo_a(\bx,\by) := 
\prod_{j\in\N} (1 + \alpha^{-1}_{1,j} \cdot m^\lo_a(x_j, y_j)) 
\]
is equal to $\X^\lo$. 
\item[(iii)] 
We have
\[
\{ f|_{\X} \colon f\in H(M^\lo_a)\} \subseteq H(K).
\]
\end{itemize}
\end{theo}

\begin{proof}
We begin by describing two embedding steps.
Let $\xi(f)$ be given by \eqref{g21}. In the first embedding 
step we consider for $j \in \N$ the reproducing kernels 
\[
\phantom{\qquad \quad x,y\in\R.}
k^\lo_{j}(x,y) := 1 + \alpha^{-1}_{1,j} \cdot h_1(x) \cdot h_1(y),
\qquad \quad x,y\in\R.
\]
Clearly,
\[
H (k^\lo_{j}) = \spann\{ h_0, h_1\}
\]
as vector spaces, 
and
\[
\|f\|^2_{H(k^\lo_j)} = 
|\xi(f)|^2
+ \alpha_{1,j} \cdot \|f\|^2
\]
for every $f \in H(k_1)$, where
\[
\|f\| := |\scp{f}{h_1}_{L^2(\mu_0)}|. 
\]
Hence $H(k^\lo_j) \subseteq H(k_j)$
with an identical embedding 
\[
T_j^{(3)} \colon H(k^\lo_j) \hookrightarrow H(k_{j})
\]
of norm one. 

For the second embedding step, we consider the reproducing 
kernel $\widecheck{m}_a$, given by
\[
\widecheck{m}_a(x,y) := (x-a) \cdot (y-a)
\]
for $x,y \in \R$. Obviously,
$H(1 + \widecheck{m}_a) = \spann\{ h_0, h_1\}$ and $\widecheck{m}_a(a,a) 
= 0$
as well as $H(1) \cap H(\widecheck{m}_a) = \{0\}$. 
Furthermore, for every $f\in\spann\{h_0,h_1\}$ there is some $y\in\R$
such that 
$f = f(a) + \widecheck{m}_a(\cdot,y)$,
which yields
\[
\|f\|_{H(1+ \gamma \widecheck{m}_a)}^2 
= f(a)^2 + \gamma^{-1}\cdot (y-a)^2
= f(a)^2 + \gamma^{-1} \cdot \scp{f}{h_1}_{L^2(\mu_0)}^2.
\]
Let $f \in \spann\{h_0, h_1\}$. Since 
\[
|\xi(f) - f(a)| \le k^\lo_1(a,a)^{1/2} \cdot \|f - \xi(f) \|_{H(k^\lo_1)} 
= ( \alpha_{1,1} +  a^2)^{1/2} \cdot \|f\|,
\]
we have
\begin{align*}
\| f \|^2_{H(k^\lo_j)} &\le 
\left( |f(a)| + |\xi(f) - f(a)|\right)^2 + 
\alpha_{1,j} \cdot \|f\|^2\\
&\le 
\left( |f(a)| + (\alpha_{1,1} + a^2)^{1/2} \cdot 
\|f\| \right)^2 + \alpha_{1,j} \cdot
\|f\|^2 \\
&\le \left(1 + \alpha_{1,j}^{-1}\right) \cdot |f(a)|^2 + 
(1 + \alpha_{1,j}) \cdot 
(\alpha_{1,1} + a^2) \cdot \|f\|^2 + 
\alpha_{1,j} \cdot \|f\|^2 \\
&\le  
\left(1 + \alpha_{1,j}^{-1} \right) \cdot
\left( |f(a)|^2 + c^\lo (a)^{-1} \cdot \alpha_{1,j} \cdot 
\|f\|^2 \right)\\
&= 
\left(1 + \alpha_{1,j}^{-1}\right) 
\cdot \| f \|^2_{H(1 + c^\lo (a) \cdot \alpha_{1,j}^{-1} \cdot
\widecheck{m}_a)}.
\end{align*}
Put $m^\lo_a := c^\lo(a) \cdot \widecheck{m}_a$. We get
\[
\| T^{(4)}_j \| \le \left(1 + \alpha_{1,j}^{-1}\right)^{1/2} 
\]
for the norm of the identical embedding
\[
T^{(4)}_j \colon H(1 + \alpha_{1,j}^{-1} \cdot m^\lo_a) 
\hookrightarrow H(k^\lo_j).
\]
Due to the equivalence of the norms $\| \cdot \|_{H(k^\lo_1)}$ and 
$\| \cdot \|_{H(1 + m^\lo_a)}$, there exists a constant $c>0$ 
such that 
\[
m^\lo_a(x,x) \le c \cdot (1 + \alpha^{-1}_{1,1} \cdot  x^2)
\]
and
\[
\alpha^{-1}_{1,1} \cdot  x^2 \le c \cdot ( 1 + m^\lo_a(x,x) )
\]
for all $x\in \R$. This in combination with (\hb) gives us 
\[
\sum_{j\in\N} \alpha^{-1}_{1,j} \cdot m^\lo_a (x_j,x_j) 
< \infty 
\qquad \Leftrightarrow \qquad
\bx \in \X^\lo.
\]
This shows that $M^\lo_a$ as defined in (ii) has maximal domain $ \X^\lo$.

Furthermore, 
\[
T^{(3)}_j \circ T^{(4)}_j \colon
H(1 + \alpha_{1,j}^{-1}\, \cdot m^\lo_a) \hookrightarrow H(k_j)
\]
is continuous with norm bounded by $(1 + \alpha^{-1}_{1,j} )^{1/2}$. Due 
to (\hb) and Theorem~\ref{ta2} applied with $u_j:=1$, 
$\X^{(\bu)} := \X^\lo$, $\Y^{(\bu)} := \X$, 
$K^{(\bu)} := M^\lo_a$, and $L^{(\bu)} := K$ yields
\[
\{ f|_\X \colon f \in H(M^\lo_a) \} \subseteq H(K).
\qedhere
\]
\end{proof}

\begin{rem}
In the cases (PG) and (EG) with $\limsup_{j \to \infty} b_j < 1$ we have
$\X \subsetneq \X^\lo$, 
which follows directly from the definition of $\X^\lo$
and \citet[Prop.~3.10 and 3.19]{GHHR2021}.
This means that the maximal domain becomes strictly larger
when switching from $H(K)$ to $H(M_a^\lo)$,
\end{rem}

\section{Integration and $L^2$-Approximation}\label{s4}

In Section \ref{s4.1} we survey and 
extend recent results from \citet{DungNguyen22} on integration 
and $L^2$-approximation in the case of univariate Hermite spaces 
of finite smoothness. We stress that in the latter paper the 
multivariate case and Banach spaces of functions with integrability 
exponent $p \in {]1,\infty[}$ are studied, while we only need to 
consider the univariate case with $p=2$.

Hermite spaces of functions of infinitely many variables
are considered in Section~\ref{s4.2} with Theorem~\ref{t3} as our main 
result under the general assumptions (\ha)--(\hc) and with 
Corollary~\ref{t1}, which deals with the particular cases (PG) and (EG).
In Section~\ref{s4.3} we describe in more detail the algorithms that 
achieve our upper error bounds.

The computational problems studied in Sections \ref{s4.1} to
\ref{s4.3} are of the following form. The function space is
an RKHS $H(L)$ on a domain $E$ that is equipped with a probability 
measure $\eta$ such that $H(L) \subseteq L^2(\eta)$.
We consider linear algorithms $A$ that use a finite number of
function values to integrate or to approximate
the functions $f \in H(L)$. More precisely,
\begin{equation}\label{form_linear_algorithmus}
\phantom{\qquad\quad f \in H(L),}
A(f) := \sum_{i=1}^m f(x_i) \cdot g_i,
\qquad\quad f \in H(L),
\end{equation}
with $m \in \N$ and $x_1,\dots,x_m \in E$, and with
$g_1, \dots, g_m \in \R$ for integration
and $g_1, \dots, g_m \in L^2(\eta)$ for $L^2$-approximation. 
The worst-case error of $A$ is defined by
\[
\err(A,L) :=
\sup
\biggl\{ \biggl|\int_E f \, d \eta - A(f) \biggr| \colon
f\in H(L),\ \|f\|_{H(L)}\leq 1 \biggr\}
\]
for integration and by
\[
\err(A,L) := \sup
\{ \|f-A(f)\|_{L^2(\eta)} \colon
f\in H(L),\ \|f\|_{H(L)}\leq 1\}
\]
for $L^2$-approximation. The worst-case cost of $A$ is defined by
\[
\cost (A) := \sum_{i=1}^m \cost(x_i),
\]
where $\cost \colon E \to {[0,\infty[} \cup \{\infty\}$ specifies
the cost of a single function evaluation.
Of course, the most natural choice of the cost function is $\cost(x) := 1$
for every $x \in E$, which means that the functions from
$H(L)$ may be evaluated anywhere in the domain $E$ at a constant
cost one. As the key quantity, we study the 
$n$-th minimal worst-case error, which is defined by
\[
\err_n(L) := \inf \{\err(A,L) \colon \cost(A) \leq n\}
\]
for $n \in \N$.

Let $\bz:=(z_n)_{n \in \N}$ denote a sequence of positive real
numbers. The decay of $\bz$ is defined by
\begin{equation}\label{g90}
\decay (\bz) := 
\sup \Bigl\{ \tau > 0 \colon \sum_{n \in \N} z_n^{1/\tau} < \infty \Bigr\}
\end{equation}
with the convention that $\sup \emptyset := 0$. As a well-known fact
\[
\decay (\bz) = 
\liminf_{n \to \infty} \frac{\ln(z_n^{-1})}{\ln(n)}
\]
if the decay or the limit inferior is positive,
see, e.g., \citet[Lem.~B.3]{GHHRW2020}. Moreover,
\[
\decay(\bz) = 
\sup\{\tau > 0 \colon 
\sup_{n \in \N} \left(z_n \cdot n^\tau \right) < \infty\},
\]
if $\bz$ is non-increasing, see, e.g., \citet[p.~250]{FHW2012}.
Henceforth we use the notation
\[
\bz^s := (z_n^s)_{n \in \N}
\]
for $s \in \R$. 

In Sections \ref{s4.1} and \ref{s4.2} we establish upper and lower
bounds for the decay
\[
\dec(L) := \decay \left( (\err_n(L))_{n \in \N} \right)
\]
of the $n$-th minimal worst-case errors.
Note that a lower bound for $\dec (L)$ corresponds to an upper bound for 
the $n$-th minimal errors $\err_n(L)$ and vice versa.

\subsection{Functions of a Single Variable}\label{s4.1}

Our setting for univariate integration and $L^2$-approximation
is specified by a non-decreasing sequence 
$\ba := (\alpha_\nu)_{\nu \in \N}$
of Fourier weights satisfying at least
\begin{equation} \label{g11}
 \sum_{\nu\in \N}  \alpha_{\nu}^{-1} \cdot \nu^{-1/2} < \infty,
\end{equation}
which yields the Hermite kernel 
\begin{equation}\label{g15}
\phantom{\qquad\quad x,y \in \R.}
k (x,y) := 1 +
\sum_{\nu\in \N} 
\alpha_{\nu}^{-1} \cdot h_\nu (x) \cdot h_\nu (y),
\qquad\quad x,y \in \R.
\end{equation}
See \citet[Sec.~3.1]{GHHR2021}. Some of the results for the
integration problem will only be established in the particular case 
$k^{[r]} := k$ with
\begin{equation}\label{g12}
\alpha_\nu := (\nu+1)^r,
\end{equation}
where we have to require $r > 1/2$ to ensure \eqref{g11}. 
Cf.\ (PG) in Section~\ref{s2.1} and recall Remark~\ref{r5}.

The measure $\eta$ on $E := \R$
that defines the integral and the $L^2$-norm 
is the univariate standard normal
distribution $\mu_0$,
and we define $\cost(x) := 1$ for every $x \in \R$.

\subsubsection{Integration}

For Fourier weights according to \eqref{g12} with $r \in \N$
even the asymptotic behavior 
\begin{equation}\label{g5}
\err_n (k^{[r]}) \asymp n^{-r}
\end{equation}
of the $n$-th minimal errors for integration is known,
see \citet{DILP18} and \citet{DungNguyen22}.
Actually, in both of these papers 
the multivariate case with the $d$-fold tensor product of the 
kernel $k^{[r]}$ and the $d$-dimensional standard normal distribution
is studied, which results in the additional factor $(\ln(n))^{(d-1)/2}$ 
in \eqref{g5}. The corresponding lower bound is due to
\citet[Thm.~1]{DILP18}, and the matching upper bound is
due to \citet[Thm.~2.3]{DungNguyen22}; upper bounds involving
a further logarithmic factor 
have already been established in \citet[Cor.~1]{DILP18} 
for $d \in \N$ and in \citet[Thm.~4.5]{KSG22}
for $d=1$.

In the sequel, we present the construction of asymptotically optimal 
quadrature formulas and a sketch of the analysis from \citet{DungNguyen22}.
To this end, we also consider the Sobolev space $W^r$ of functions on the
interval $I := [-1/2,1/2]$ with weak derivatives up to order
$r$ in $L^2(\lambda)$, where $\lambda$ denotes the Lebesgue 
measure on $I$. The corresponding norm is given by
\[
\|f\|_{W^r}
:= 
\left(\sum_{i=0}^r \|f^{(i)}\|_{L^2(\lambda)}^2\right)^{1/2}
\]
for $f \in W^r$.

For $f \colon \R \to \R$ and $\ell \in \Z$
we use $f(\cdot + \ell)$ to denote the function $x \mapsto
f(x+\ell)$ on the domain $I$. Put
\[
\phi(x) := (2\pi)^{-1/2} \cdot \exp(-x^2/2)
\]
for $x \in \R$. 
The basic idea is to only take into account finitely many
integer shifts $I + \ell$ of $I$, instead of $\R$,
and to apply asymptotically
optimal quadrature formulas on $W^r$ to the integrands
$(f \phi) (\cdot + \ell)$.

See \citet[Eqn.~2.11]{DungNguyen22} for the following key result.

\begin{lemma}\label{la1}
For every $r \in \N$ and every
$\delta \in {]0,1/4[}$ there exists $c > 0$ such that
\[
\|(f \phi)(\cdot+\ell)\|_{W^r} \leq 
c \cdot \exp(- \delta \ell^2) \cdot\|f\|_{H(k^{[r]})}
\]
for every $f \in H(k^{[r]})$ and every $\ell \in \Z$.
\end{lemma}

It is well known that the $n$-th minimal errors for integration
on $W^r$ with respect to $\lambda$ are of the order $n^{-r}$.
In the sequel, we employ any asymptotically optimal sequence 
of $m$-point quadrature formulas $A^\prime_m$ on $I$, i.e.,
there exists $c > 0$ such that
\[
\left| \int_I f \, d \lambda - A^\prime_m(f) \right|
\leq c \cdot m^{-r} \cdot \|f\|_{W^r}
\]
for every $f \in W^r$ and every $m \in \N$. 
This optimality holds, for instance, for quadrature
formula based on equidistant nodes from $I$ and piecewise
polynomial interpolation of degree at least $r-1$.

For every $L \in \N$ and every sequence $\bm := (m_\ell)_{|\ell| < L}$
in $\N$ we obtain a quadrature formula
\[
A_{L,\bm} (f) := \sum_{|\ell| < L} 
A^\prime_{m_\ell} ((f \cdot \varphi) (\cdot + \ell))
\]
on $H(k^{[r]})$.

See \citet[p.~8]{DungNguyen22} for the following fact, which
immediately follows from Lemma~\ref{la1} and the continuity of
$f \mapsto \int_I f \, d \lambda$ on $W^r$.

\begin{lemma}\label{la2}
For every $r \in \N$ and every $\delta \in {]0,1/4[}$
there exists $c > 0$ such that
\[
\err(A_{L,\bm}, k^{[r]})
\leq c \cdot \left(
\sum_{|\ell| < L} m_\ell^{-r} \cdot \exp (-\delta \ell^2) 
+
\sum_{|\ell| \geq L} \exp (-\delta \ell^2) \right)
\]
for all $L \in \N$ and $\bm$ as before.
\end{lemma}

In the sequel, we fix some $\delta := {]0,1/4[}$.
For $n \geq 2$ we choose
\[
L_n := 
\left\lceil \left( 
\frac{r}{\delta} \cdot \ln(n) \right)^{1/2} \right\rceil
\]
and
\[
m_{\ell,n} := 
\left\lceil n \cdot \exp
\left(-\frac{\delta}{2r} \cdot \ell^2\right) \right\rceil,
\]
and we use $A_n := A_{L_n,\bm_n}$ to denote the corresponding quadrature 
formula on 
$H(k^{[r]})$.

See \citet[Thm.~2.1]{DungNguyen22} for the following result,
which yields the upper bound in \eqref{g5}
and is derived from Lemma~\ref{la2} in a straightforward way.

\begin{theo}\label{ta4}
For the integration problem on $H(k^{[r]})$ the following holds 
true.
For every $r \in \N$ there exists $c > 0$ such that the worst-case 
error and the number of nodes of $A_n$ satisfy
\[
\err(A_n,k^{[r]})
\leq c \cdot n^{-r}
\]
and
\[
\sum_{|\ell| < L_n} m_{\ell,n} \leq c \cdot n,
\]
respectively, for every $n \in \N$.
\end{theo}

\begin{rem}\label{r10}
Let $\overline{r} \in \N$.
Since $L_n$ and $m_{\ell,n}$ are non-decreasing functions of $r$,
we may easily construct a sequence of quadrature formulas
that is asymptotically optimal simultaneously for all
$r \in \{1,\dots,\overline{r}\}$.
\end{rem}

The result from \citet{DungNguyen22} may be extended in the
following way.

\begin{theo}\label{t5}
For the integration problem on $H(k)$ the following holds true.
\begin{itemize}
\item[{\rm (}i{\rm )}] If $\decay (\ba^{-1}) > 1$ then
\[
 \dec (k) \geq \decay (\ba^{-1}).
\]
\item[{\rm (}ii{\rm )}] For Fourier weights of the form 
\eqref{g12} we have
\[
 \dec (k^{[r]}) = \decay (\ba^{-1}) = r
\]
if $r \geq 1$ and
\[
 \dec (k^{[r]}) \leq \decay (\ba^{-1}) = r
\]
if $1/2 < r < 1$.
\end{itemize}
\end{theo} 

\begin{proof}
At first, we consider Fourier weights according to 
\eqref{g12}, where we obtain 
\[
\dec (k^{[r]}) = r
\]
for $r \in \N$ immediately from \eqref{g5}.

It has been observed in 
\citet[Rem.~2.3 and Exmp.~3.5]{GHHR2021} that 
the spaces $H(k^{[r]})$ with $r>1/2$ form 
an interpolation scale with respect to quadratic 
interpolation by means of the $K$-method as well as the $J$-method. 
Since optimal algorithms for integration on $H(k^{[r]})$ are linear 
algorithms and the worst-case error of a linear algorithm is the norm of 
a linear functional on $H(k^{[r]})$, also the worst-case errors for 
integration with a fixed linear algorithm interpolate. This carries 
then over to the errors $\err_{n} (k^{[r]})$. 

More precisely, let $1/2 < \underline{r} < r < \overline{r}$. 
Then there 
exists a uniquely determined $ \theta \in {]0,1[}$ such that 
\[
r = (1-\theta) \underline{r} + \theta \overline{r}.
\] 
Consequently, the Fourier weights of the reproducing kernel Hilbert 
spaces $H(k^{[r]})$, $H(k^{[\underline{r}]})$, and
$H(k^{[\overline{r}]})$ satisfy 
\[
(\nu +1)^r = \left( (\nu + 1)^{\underline{r}} \right)^{1-\theta} 
\cdot 
\left( (\nu + 1)^{\overline{r}} \right)^{\theta},
\]
and we get from the quadratic interpolation method
\[
\err (A,k^{[r]}) \le 
\err (A,k^{[\underline{r}]})^{1-\theta} \cdot \err
(A,k^{[\overline{r}]})^{\theta}
\]
for every quadrature formula $A$, 
cf.\ \citet[Rem.~2.3 and Exmp.~3.5]{GHHR2021}. 
Using Remark~\ref{r10}, we obtain
\[
 \err_{n} (k^{[r]}) \le 
\err_{n} (k^{[\underline{r}]})^{1-\theta} \cdot \err_{n}
(k^{[\overline{r}]})^{\theta}.
\]
This further implies 
\begin{equation} \label{eq:interpol}
\dec (k^{[r]}) \ge  (1-\theta) \dec(k^{[\underline{r}]}) + \theta
\dec(k^{[\overline{r}]}).
\end{equation}
Hence, for arbitrary $r>1$ we may choose
$\underline{r},\overline{r} \in \N$ with 
$\underline{r}<r<\overline{r}$ to obtain 
\[
\dec (k^{[r]}) \ge r.
\]

On the other hand, using \eqref{eq:interpol} in the equivalent form
\[
 \dec(k^{[\underline{r}]}) \le \frac{1}{1-\theta} \cdot 
\left( \dec(k^{[r]}) - \theta \dec(k^{[\overline{r}]}) \right)
\]
with $r,\overline{r} \in \N$ and $\underline{r}>1/2$ 
shows that 
\[
\dec (k^{[\underline{r}]}) \le \underline{r}
\]
for every $\underline{r}>1/2$.

Finally, we consider any non-decreasing sequence $\ba$ of Fourier
weights with $\decay(\ba^{-1}) > 1$. For every $1 \leq r <
\decay(\ba^{-1})$ the space $H(k)$ is continuously embedded into
$H(k^{[r]})$, and therefore
\[
\dec(k) \geq \dec(k^{[r]}) = r
\]
which implies $\dec(k) \geq \decay (\ba^{-1})$.
\end{proof}

\subsubsection{$L^2$-Approximation}

For Fourier weights of the form \eqref{g12} 
the asymptotic behavior of the $n$-th minimal errors 
for $L^2$-approximation is even known for all $r > 1$, where we have 
\begin{equation}\label{g20}
\err_n (k^{[r]}) \asymp n^{-r/2},
\end{equation}
see \citet[Thm.~3.5]{DungNguyen22}. 

As for integration, the multivariate case with the $d$-fold tensor 
product of the kernel $k^{[r]}$ and the $d$-dimensional standard normal 
distribution is studied for $L^2$-approximation, too, in 
\citet{DungNguyen22}, which results 
in the additional factor
$(\ln(n))^{(d-1) \cdot r /2}$ in \eqref{g20}.

In the further analysis of the approximation problem
we use recent general results on upper bounds for the 
minimal worst-case error for $L^2$-approxima\-tion on separable
RKHSs using function values. 
The most recent and sharp estimates, which
are also employed in \citet{DungNguyen22}, are from \citet{DKM22}. 
For our estimates of $\dec(k)$ it would be also sufficient 
(but perhaps less elegant) to use earlier results, which are suboptimal 
only by some logarithmic factors, see \citet{DKM22} for references.

Before stating the corresponding theorem, 
we provide an auxiliary lemma concerning the decay of
$\bb := (\beta_n)_{n \in \N}$ given by 
\[
\beta_n := 
\left( \frac{1}{n} \sum_{\nu\ge n} \alpha_\nu^{-1} \right)^{1/2}
\]
under the assumption
\begin{equation} \label{eq:summability}
\sum_{\nu\in \N}  \alpha_{\nu}^{-1} < \infty.
\end{equation}

\begin{lemma}\label{Lemma:Auxiliary_t4}
If \eqref{eq:summability} is satisfied then
\begin{equation*}
\decay (\bb) = \decay( \ba^{-1/2}). 
\end{equation*}
\end{lemma}

\begin{proof}
Due to \eqref{eq:summability}, we have $\decay( \ba^{-1}) \ge 1$.

The weak discrete Stechkin inequality assert that for $q>1$ 
there exist 
constants $c,C>0$ depending only on $q$ such that for all
non-increasing sequences 
$(\omega_\nu)_{\nu \in \N}$ of positive real numbers we have
\[
 c^{-1} \sup_{n \ge 1} n 
\left( \frac{1}{n} \sum_{\nu \ge n} \omega_\nu^q \right)^{1/q}
 \le \sup_{n \ge 1} n \omega_n \le 
 C \sup_{n \ge 1} n 
\left( \frac{1}{n} \sum_{\nu \ge n} \omega_\nu^q \right)^{1/q}.
\]  
For this and related inequalities including optimal constants,
see \citet{JU21}.
Applying this two-sided estimate with $q := 2 \tau$ and 
$\omega_\nu := \alpha_\nu^{-1/(2\tau)}$ shows that 
$\decay( \ba^{-1/2}) > \tau$ if and only if 
$\decay( \bb ) > \tau$, provided that $\tau > 1/2$.
Hence  $\decay( \bb ) = \decay( \ba^{-1/2})$ follows in the case 
$\decay( \ba^{-1}) > 1$.

In the borderline case $\decay( \ba^{-1}) = 1$ we use
\[
\beta_n^2 \le \frac{1}{n} \sum_{\nu \in \N} \alpha_\nu^{-1}
\]
and \eqref{eq:summability} to conclude that 
$\decay (\bb) \ge 1/2 = \decay( \ba^{-1/2})$.
As observed above, $\decay (\bb) > 1/2$ would imply 
$\decay( \ba^{-1/2}) > 1/2$. Hence 
$\decay (\bb)= 1/2 = \decay( \ba^{-1/2})$.
\end{proof}

\begin{theo}\label{t4}
For $L^2$-approximation on $H(k)$ the following holds true.
If \eqref{g11} is satisfied then
\[
 \dec (k) \leq \decay (\ba^{-1/2}).
\]
If \eqref{eq:summability} is satisfied then 
\[
 \dec (k) = \decay (\ba^{-1/2}).
\]
\end{theo}

\begin{proof}
The sequence of singular values of the embedding of $H(k)$ 
into $L^2(\mu_0)$ is $ \ba^{-1/2}$.
Since the singular values, as the minimal worst-case errors of 
$L^2$-approximation using general linear information, are lower bounds 
for the minimal worst-case errors of $L^2$-ap\-proxi\-ma\-tion using 
function values, more precisely, since 
\[
\err_{n} (k) \ge \alpha_n^{-1/2}
\]
for all $n\in \N$, the inequality 
$\dec (k) \le \decay (\ba^{-1/2})$ follows. 
 
For the reverse inequality, we use \citet[Thm.~1]{DKM22},
which shows that there exists a universal constant $c \in \N$ such that
\[
\err_{c n} (k) \le  \beta_n
\] 
for every separable RKHS $H(k)$ with square-summable singular
values $\alpha_\nu$ of its identical embedding into any $L^2$-space.
In the present case the square-summability is guaranteed by
\eqref{eq:summability}.
Since the $n$-th minimal errors $\err_n(k)$ form a non-increasing
sequence, the decay $ \dec (k)$ of   
$(\err_{n} (k))_{n\in \N_0}$ is the same as the decay of the 
subsequence  $(\err_{cn} (k))_{n\in \N_0}$. 
We conclude that $\dec (k) \geq \decay (\bb)$, and it
remains to observe Lemma~\ref{Lemma:Auxiliary_t4}.
\end{proof}

\begin{rem}\label{r3}
The approach yielding the error bound from \citet[Thm.~1]{DKM22}
is based on a least square estimator using independent random sample 
points drawn with respect to a suitable density. Thus, the proof is 
non-constructive, as it only ensures the existence of a good
deterministic algorithm using function values. Of course,
it would be desirable 
to have explicit constructions for sample points achieving this 
univariate error decay.
\end{rem}

\subsection{Functions of Infinitely Many Variables}\label{s4.2}

Now we consider integration and $L^2$-approximation 
with respect to the product measure $\mu$ of the univariate 
standard normal distribution $\mu_0$ for functions from the 
Hermite space $H(K)$, which depend on infinitely many variables.
As before, we assume that the corresponding Fourier weights satisfy
(\ha)--(\hc), which implies $\mu(\X)=1$ for
the maximal domain $\X$ of $K$. Formally, we consider 
the restriction $\eta$ of $\mu$ onto $E:= \X$.

In our analysis of algorithms we employ the unrestricted subspace 
sampling model, which has been introduced in \citet{KuoEtAl10}. 
This model is based on a non-decreasing cost function 
$\$\colon \N_0 \to \left[1,\infty\right[$ and some 
nominal value $a\in \R$ in the following way. For $\bx \in \R^\N$ the 
number of active variables is defined by
\[
\Act_a(\bx) :=\#\{j\in\N \colon x_j \neq a\},
\]
and the evaluation of any function $f \in H(K)$ is permitted at
any point 
\[
\bx \in \X_a := \{ \bx \in \R^\N \colon \Act_a(\bx) < \infty\}
\subsetneq \X;
\]
the corresponding cost is given by $\$ (\Act_a(\bx))$.
Accordingly, we define
\[
\cost(\bx) := 
\begin{cases}
\$(\Act_a(\bx)) & \text{if $x \in \X_a$,}\\
\infty & \text{otherwise}
\end{cases}
\]
for $x \in \X$.
Throughout this paper we assume that 
there exist $c_1,c_2 > 0$ such that
\begin{equation}\label{g82}
c_1 \cdot n \leq \$(n) \leq \exp(c_2 \cdot n)
\end{equation}
for all $n \in \N$,
cf., e.g., \citet{KuoEtAl10} and \citet{PW11}.
The lower bound in \eqref{g82} is most reasonable,
since $\Act_a(\bx)$ real numbers are already needed to specify any
point $\bx \in \X_a$, and the upper bound is very generous.

In the following theorem we present an upper bound and a lower bound 
for the decay $\dec(K)$ of the $n$-th minimal errors $\err_n(K)$. 
We use $\ba_1^{-1}$ and $\bg$ to denote the sequences
$(\alpha^{-1}_{1,j})_{j\in \N}$ and $(\gamma_j)_{j \in \N}$,
which satisfy $\decay(\ba_1^{-1}) \geq 1$ and $\decay(\bg) \geq 1$,
see (\hb) and (\hc).

\begin{theo}\label{t3}
For integration and $L^2$-approximation we have
\[
\min \left( \dec (k_1), \frac{\decay(\bg) - 1}{2} \right)
\le \dec (K) \le 
\min \left( \dec (k_1), \frac{\decay(\ba_1^{-1}) - 1}{2} \right).
\]
\end{theo}

\begin{proof}
Let us consider any fixed $a \in \R$.

Due to Theorem~\ref{t2}, we have $H(k_1) = H(1 + m^\up_a)$. Hence the 
closed graph theorem ensures that $H(k_1)$ is continuously embedded into 
$H(1 + m^\up_a)$ and vice versa. Consequently, 
\[
\dec(k_1) = \dec(1+m_a^\up).
\]
Furthermore, due to Theorem~\ref{t2},
the restriction map 
$f\mapsto f|_{\X^\up}$ is a continuous linear map from $H(K)$ into 
$H(M^\up_a)$.
Due to Lemma~\ref{l22}, 
we have $f = f|_{\X^\up}$ in $L^2(\mu)$ for all 
$f\in H(K)$ and hereby
\[
\err_n(K) \leq C \cdot \err_n(M^\up_a)
\]
for all $n \in \N$, 
with $C$ denoting the norm of the restriction map.
Therefore the lower bound on $\dec(K)$ follows from
\[
\dec(M^\up_a) \ge \min \left( \dec (1+m^\up_a),
\frac{\decay(\bg) - 1}{2} \right),
\]
a result that was established for superpositions of weighted tensor
products of an anchored kernel, 
see \citet[Thm.~2]{PW11} for integration and \citet[Cor.~9]{Was12} for 
$L^2$-approximation and cf.\ Remark~\ref{r26}.
See Section~\ref{s4.3} for further details.

Due to Theorem~\ref{Lemma:Tensorprodukt2_dim_Raeume} and
the closed graph theorem, we have 
that the restriction $f\mapsto f|_{\X}$ is a continuous linear map from 
$H(M^\lo_{a})$ into $H(K)$. Since $\mu(\X) = \mu(\X^{\lo}) =1$, 
we obtain 
$f = f|_{\X}$ in $L^2(\mu)$ for all $f\in H(M^\lo_{a})$. Consequently, 
\[
\err_n(K) \geq c \cdot \err_n(M^\lo_a)
\]
for all $n\in \N$,
with $c$ denoting the norm of the restriction map.
Since integration is not harder than 
$L^2$-approximation, the upper bound 
\[
\dec(M^\lo_a) \le \frac{\decay(\ba_1^{-1}) - 1}{2}
\]
follows now from \citet[Sec.~3.3]{KuoEtAl10}, cf.\ also 
\citet[Thm.~4 and Sec.~3.2.1]{DG12}. 
Note that additionally 
\[
\err_n(K) \ge \err_n(k_1)
\]
for all $n\in\N$. This establishes the upper bound on $\dec(K)$ in 
Theorem~\ref{t3}.
\end{proof}

Theorem~\ref{t3}, together with the results
from Section~\ref{s4.1} for the univariate case,
yields matching upper and lower bounds on the decay $\dec(K)$
in the cases (PG) and (EG), i.e., for Fourier weights with a
polynomial and with a (sub-)exponential growth, respectively.
We put
\[
\rho := \liminf_{j \to \infty} \frac{r_j \cdot \ln(2)}{\ln(j)} 
\geq 1,
\]
which quantifies the asymptotic behavior of 
$\alpha_{1,j}$ as $j$ tends to $\infty$.

\begin{corollary}\label{t1}
In the case (PG) we have
\[
\dec (K) = \tfrac{1}{2} \cdot \min\left(2r_1, \rho - 1\right)
\]
for the integration problem and
\[
\dec (K) = \tfrac{1}{2} \cdot \min\left(r_1, \rho - 1\right)
\]
for the $L^2$-approximation problem. 
In the case (EG) we have 
\[
\dec (K) =
\tfrac{1}{2} \cdot \left(\rho - 1\right)
\]
for the integration and the $L^2$-approximation problem.
\end{corollary}

\begin{proof}
We have
\[
\gamma_j = 2^{r_1-r_j} = 2^{r_1} \cdot \alpha_{1,j}^{-1}
\]
for every $j \in \N$
in both cases, (PG) and (EG), see Lemma~\ref{l10}, and therefore
\begin{equation}\label{eq:rho_dec}
\decay(\bg) = \rho = \decay(\ba_1^{-1}). 
\end{equation}
Together with Theorem~\ref{t3} this implies
\[
\dec (K) = 
\min \left( \dec (k_1), \frac{\rho - 1}{2} \right)
\]
for integration and for $L^2$-approximation.

The decay of
$(\alpha_{\nu,1}^{-1})_{\nu \in \N}$ is equal to $r_1$ in the case
(PG) and equal to $\infty$ in the case (EG).
Using Theorems~\ref{t5} and \ref{t4} we obtain
$\dec(k_1) = r_1$ for integration and $\dec(k_1) = r_1/2$ for 
$L^2$-approximation in the case (PG), while $\dec(k_1) = \infty$
for both problems in the case (EG).
\end{proof}

\begin{rem}
Corollary~\ref{t1} reveals that $r_1$, which is the minimal smoothness
among all Hermite spaces $H(k_j)$ of univariate functions, and
$\rho$, which concerns the growth of the smoothness as $j \to
\infty$, determine the decay of the minimal errors
on Hermite spaces $H(K)$ with
Fourier weights of a polynomial or (sub-)exponential growth.
On the one hand, if $\rho$ is sufficiently large then these
minimal errors decay as fast as the minimal errors in the
univariate case on the space
$H(k_1)$; for (PG) this means $\rho \geq 2r_1+1$ or 
$\rho  \geq r_1+1$, while we need $\rho=\infty$ in the case (EG).
On the other hand, we have $\dec(K)=0$ if $\rho=1$.
\end{rem}

\subsection{Multivariate Decomposition Methods on Hermite Spaces}%
\label{s4.3}

Let us now describe the algorithms that yield the lower bounds on the 
polynomial decay rate $\dec(K)$ of the $n$-th minimal errors for
integration and $L^2$-approxi\-ma\-tion on Hermite spaces
of functions depending on infinitely many variables. These kind of 
algorithms were first called \emph{changing dimension algorithms}, 
see \citet{KuoEtAl10},  and are now known as 
\emph{multivariate decomposition methods (MDMs)}. MDMs rely on the 
\emph{anchored function decomposition}, which is also known as 
\emph{cut HDMR}, where HDMR stands for ``high-dimensional model 
representation'', cf.\ \citet{RA99} and \citet{KSWW10}.

\subsubsection{MDMs operating on $H(M^\up_a)$}\label{s4.3.1}

At first, we consider 
the reproducing kernels $m_a^\up$ and $M^\up_a$ 
with a fixed $a\in \R$, as introduced in 
Theorem~\ref{t2} and further discussed in Remark~\ref{r26}. Since 
\[
H(1) \cap H(m_a^\up) = \{0\} \qquad \text{and} \qquad
M_a^\up = \sum_{\bu \in \bU} \gamma_\bu \cdot m_{a,\bu}^\up,
\]
the Hilbert space $H(M_a^\up)$ is the orthogonal sum of
the closed subspaces $H(m_{a,\bu}^\up)$ with $\bu \in \bU$.
See \citet[Sec.~3.2]{GMR14} for a general result. 
The construction of MDMs is based on the corresponding decomposition
\begin{equation}\label{g81}
f = \sum_{\bu \in \bU} f_\bu
\end{equation}
of the functions $f \in H(M^\up_a)$, where
$f_\bu$ denotes the orthogonal projection of $f$ onto 
$H(m^\up_{a,\bu})$, so that the series in
\eqref{g81} converges in $H(M_a^\up)$.
Observe that $f_\bu$ only depends on the variables with indices in
$\bu$, and
\begin{equation}\label{g77}
\|f\|^2_{H(M_a^\up)} = \sum_{\bu \in \bU} \gamma_\bu^{-1} \cdot
\|f_\bu\|^2_{H(m_{a,\bu}^\up)}.
\end{equation}

Since $m^\up_a$ is anchored at $a$, i.e.,
$m_a^\up(a,a) = 0$, we have
\begin{equation}\label{g79}
f_\bu(\bx) = 
\sum_{\bv\subseteq \bu} (-1)^{|\bu\setminus \bv|} \cdot
f(\bx_\bv, \bsa_{\bv^c})
\end{equation}
for every $\bx \in \X^\up$, where $f(\bx_\bv, \bsa_{\bv^c})$
denotes the value of $f$ at the point $\by$
given by $y_j := x_j$ if $j\in \bv$ and $y_j=a$ otherwise, 
see \citet[Exmp.~2.3]{KSWW10}.
Observe that $\by \in \X^\up$ due to Lemma~\ref{l22}.
In particular, $f_\emptyset$ is constant and
equal to $f(\bsa)$, where $\bsa := (a,a,a, \dots) \in \X^{\up}$.
We conclude that a finite number of function values of $f$ suffices
to obtain a function value of $f_\bu$; the corresponding number 
$2^{|\bu|}$ is acceptable as long as $|\bu|$ is sufficiently small.
The decomposition \eqref{g81} with $f_\bu$ given by
\eqref{g79} is called the \emph{anchored decomposition} of $f$.
The anchored component $f_\bu$ of $f$ has the property
that $f_\bu(\bx_\bu) = 0$ if $x_j = a$ for some $j\in \bu$. 
See \citet[Thm.~2.1]{KSWW10} for a general result on the
decomposition of multivariate functions.

In principle,
we may study integration and $L^2$-ap\-proxi\-mation on each of
the spaces $H(m^\up_{a,\bu})$ with $\bu \in \bU \setminus \{\emptyset\}$.
Formally the underlying measure is the infinite product $\mu$, 
but since the elements of $H(m^\up_{a,\bu})$ only depend on the 
variables with indices in $\bu$, we are actually dealing with 
integration and $L^2$-approximation with respect to the 
$|\bu|$-dimensional standard normal distribution.

To construct an MDM, we have to choose 
a finite set $\mathcal{A}$ of non-empty elements of $\bU$, 
i.e., of finite sets of variables,
and for each $\bu \in \mathcal{A}$ an algorithm 
$A_{\bu,n_\bu}$ for integration or $L^2$-approximation 
of functions from $H(m_{a,\bu}^\up)$,
which uses $n_\bu \in \N$ function values of each input function
$f_\bu \in H(m^\up_{a,\bu})$.
We assume that the algorithms $A_{\bu,n_\bu}$ are of the form 
\eqref{form_linear_algorithmus}.
For $L^2$-approximation we assume additionally
that also
the approximating functions $A_{\bu,n_\bu}(f_\bu)$ only depend on the 
variables with indices in $\bu$.
The corresponding MDM on $H(M^\up_a)$ is given by
\begin{equation}\label{eq:form_MDM}
A (f) := 
f(\bsa) + 
\sum_{\bu \in \mathcal{A}} A_{\bu,n_\bu}(f_\bu).
\end{equation}
For notational convenience we put $n_\bu := 0$ 
for $\bu \in \bU \setminus (\mathcal{A} \cup \{ \emptyset\})$
and $A_{\bu,0} := 0$. 
For every choice of $\mathcal{A}$ and of algorithms
$A_{\bu,n_\bu}$ for $\bu \in \mathcal{A}$ we obtain 
\begin{equation}\label{g80}
\err^2(A,M^\up_a) 
\leq
\sum_{\bu \in \bU \setminus \{\emptyset\}} 
\gamma_\bu \cdot \err^2 (A_{\bu,n_\bu},m^\up_{a,\bu})
\end{equation}
from \eqref{g77}.
See \citet[p.~513]{PW11} for integration;
the same argument applies for $L^2$-approximation.
We add that
\[
\err (A_{\bu,0},m^\up_{a,\bu}) =
\left(\err (0,m^\up_a)\right)^{|\bu|}
\]
for the error of the zero algorithm, i.e., for the operator norm of the 
integration functional or the $L^2$-embedding operator.
Since $A_{\bu,n_\bu}$ is applied to the anchored
components $f_\bu$ of $f \in H(M_a^\up)$ in \eqref{eq:form_MDM}, 
we obtain 
\begin{equation}\label{g83}
\cost(A) \leq \$(0) +
\sum_{\bu \in \mathcal{A}} n_\bu \cdot 2^{|\bu|} \cdot \$(|\bu|)
\end{equation}
from \eqref{g79}, see, e.g., \citet[p.~512]{PW11}.

Since $m^\up_{a,\bu}$ is of tensor product form,
Smolyak's construction may be used to obtain good algorithms 
$A_{\bu,n_\bu}$ 
on the space $H(m^\up_{a,\bu})$ from good algorithms 
on the space $H(m_a^\up)$ of functions of
a single variable. Most importantly,
the results from \citet{WW95} or \citet{GW2020} yield
explicit upper bounds 
for the error of Smolyak algorithms
$A_{\bu,n_\bu}$ on the spaces $H(m_{a,\bu}^\up)$
for all $\bu \in \bU \setminus \{\emptyset\}$ and $n_\bu \in \N$.

For suitable MDMs we thus 
have an explicit upper bound for the worst-case error on 
the unit ball in $H(M_a^\up)$ and an
explicit upper bound for the cost. See \citet{PW11} and
\citet{Was13} for a detailed
analysis in a general setting and for the almost optimal choice of 
$\mathcal{A}$ and of the algorithms $A_{\bu,n_\bu}$ for $\bu \in
\mathcal{A}$.

\subsubsection{MDMs operating on $H(K)$}

Now we actually want to find algorithms for integration or
$L^2$-approximation on the Hermite space $H(K)$. To this purpose, we 
clearly prefer to work exclusively with the Hermite kernels $k_1$ and $K$ 
and to construct MDMs without having to calculate the reproducing 
kernels $m_a^\up$ and $M^\up_a$ explicitly. Let us sketch the idea of 
how to achieve this goal: Basically, we only
have to find a good sequence 
of algorithms $A_n$ for integration or $L^2$-approximation, 
respectively, on the Hermite space $H(k_1)$ of univariate
functions,
which use $n$ function values of each input $f \in H(k_1)$.
Although we do not know $m^\up_a$ explicitly,  we still know that 
$H(m^\up_a)$ is continuously embedded in $H(k_1)$ and that the 
corresponding embedding constant is at most $c^\up(a)$,
see Theorem~\ref{t2}(i).
Hence we have
\[
\err(A_n, m_a^\up) \leq c^\up(a) \cdot \err(A_n, k_1),
\]
which serves as the starting point for the analysis of
the Smolyak algorithms $A_{\bu,n_\bu}$ on the spaces
$H(m_{a,\bu}^\up)$. Furthermore,
\[
\err(0,m_a^\up) \leq c^\up(a)
\]
for integration and
\[
\err(0,m_a^\up) \leq c^\up(a) \cdot \max(1,\alpha_{1,1}^{-1/2})
\]
for $L^2$-approximation.
Finally, we employ Theorem~\ref{t2}.(iii)
to return to the space $H(K)$:
Since the embedding of $H(K)$ into $H(M^\up_a)$ by restriction 
is continuous with operator norm bounded by $C^\up(a)$, we obtain
\begin{equation}\label{g84}
\err^2(A,K) 
\leq \left(C^\up(a)\right)^2 \cdot 
\sum_{\bu \in \bU \setminus \{\emptyset\}} \gamma_\bu \cdot \err^2
(A_{\bu,n_\bu},m^\up_{a,\bu})
\end{equation}
from \eqref{g80}.

We stress that the only difference between the standard setting
for the MDM, as discussed in Section~\ref{s4.3.1},
and the Hermite setting is the presence of
$c^\up(a)$ and $C^\up(a)$ in the error bounds, with small values of 
both quantities being preferable. We have explicit formulas for
$c^\up(a)$ and $C^\up(a)$, and the minimal value of both of
these quantities is attained for $a=0$, see Remark~\ref{r26}.

\subsubsection{The Cases {\rm (PG)} and {\rm (EG)}}

Let us become more specific for the cases (PG) and (EG), 
where we assume $r_1 \leq r_2 \dots$ to simplify the presentation
and $\rho > 1$ as well as $\dec(k_1) > 0$ to exclude the trivial
case $\dec(K)=0$. We add that
\[
\gamma_\bu = 2^{\sum_{j \in \bu} (r_1-r_j)}
\]
in both cases, see Lemma~\ref{l10}.

At first, we consider univariate integration and
$L^2$-approximation, as studied in Section~\ref{s4.1}.
Let us assume that we have constants $c, \kappa > 0$
and a sequence of algorithms $A_n$ on $H(k_1)$ 
using $n$ function evaluations and satisfying
\begin{equation}\label{est:pre_Smolyak_bound} 
\err(A_n, k_1) \le c \cdot n^{-\kappa}
\end{equation}
for every $n \in \N$.
Clearly, $\kappa$ cannot be larger than $\dec(k_1)$, and in view of 
Corollary~\ref{t1} we may confine ourselves to 
\[
\kappa < \min \left( \dec(k_1), \frac{\rho-1}{2} \right).
\]
Note that 
\[
\err(A_n, m^\up_a) \le c \cdot c^\up(a) \cdot n^{-\kappa}
\qquad \text{and} \qquad \err(0, m^\up_a) \le c^\up(a).
\]

Based on the algorithms $A_n$, Smolyak's
construction provides us with linear 
algorithms $A_{\bu,n}$ that use at most $n$ function values of any 
function from $H(m_{a,\bu}^\up)$. The following
statement can be found in \citet[Thm.~7]{GW2020}
(where not only deterministic, but, more generally, randomized algorithms 
are considered) or can be derived from \citet[Thm.~1]{WW95}. 
Given \eqref{est:pre_Smolyak_bound}, there exist constants
$C_0,C_1 > 0$, which only depend 
on $c$, $c^\up(a)$, and $\kappa$, such that
\begin{equation}\label{g85}
\err(A_{\bu,n}, m^\up_{a,\bu}) \le 
C_0 C_1^{|\bu|} 
\left( 1 + 
\frac{\ln(n+1)}{\max(|\bu|-1,1)} \right)^{(\kappa + 1)(|\bu|-1)} 
(n+1)^{-\kappa}
\end{equation}
for all $\bu \in \bU \setminus \{\emptyset\}$ and $n  \in \N_0$.
Together with \eqref{g84}, this yields 
an explicit upper bound for the worst-case error of the MDM $A$ 
on the unit ball in $H(K)$ in
terms of $\mathcal{A}$ and of $n_\bu$ for $\bu \in \mathcal{A}$.

We may therefore consider the following optimization problem: 
For a given error tolerance $\eps > 0$, determine 
$\mathcal{A}$ and $n_\bu$ for $\bu \in \mathcal{A}$ such that 
the upper bound for $\err(A,K)$ is at most $\eps$ and
the upper bound \eqref{g83} for $\cost(A)$ is as small as possible.
This problem has been studied in a more general setting
in, e.g., \citet{PW11} and \citet{Was13}.

In the sequel, we follow \cite{PW11}.
We fix $\eps > 0$ and
\[
\delta \in 
\frac{1}{\rho} \cdot  {]2\kappa,\rho-1[}
\]
and put
\[
L := \prod_{j \in \N} (1+ \gamma_j^{1-\delta}) - 1.
\]
Since $\decay(\bg) = \rho$, see \eqref{eq:rho_dec},
this choice of $\kappa$ and $\delta$ ensures that $L$ as well as 
$\sum_{j \in \N} \gamma_j^{\delta/(2\kappa)}$ are finite. 
Moreover, we define
\[
p_\bu := C_1^{2|\bu|} \cdot \gamma_\bu^ \delta.
\] 
Finally, we choose
\[
\mathcal{A} := 
\{ \bu\in \bU \colon p_\bu \cdot L C_0^2
> \varepsilon^2 \}
\]
and 
\[
n_\bu := 
\left\lfloor  
\left( p_\bu \cdot L C_0^2 \cdot \eps^{-2} \right)^{1/(2\kappa)}
\right\rfloor
\]
for every $\bu \in \mathcal{A}$.
Due to \eqref{g84} and \eqref{g85}, the algorithm $A$ from 
\eqref{eq:form_MDM} satisfies the error bound
\[
\err(A, K) \le C^\up(a) \cdot 
B(\eps) \cdot \eps,
\]
where
\[
B(\eps) :=
\sup_{\bu \in \mathcal{A}} 
\left( 1 + 
\frac{\ln(n_\bu+1)}{\max(|\bu|-1,1)} \right)^{(\kappa + 1)(|\bu|-1)}.
\]
Due to \eqref{g83}, we obtain the cost bound
\[
\cost(A) \le \$(0) + C_2 \cdot
\$(d(\eps)) \cdot \eps^{-1/\kappa},
\]
where 
\[
C_2 := 
(C_0 L^{1/2})^{1/\kappa} \cdot
\exp \left( 2 C_1^{1/\kappa} 
\sum_{j \in \N} \gamma_j^{\delta/(2\kappa)} \right)
\]
and 
\[
d(\eps):= \max\{|u| \colon u\in \mathcal{A}\} 
\]
is the so-called active dimension. Since 
\[
B(\eps) = \eps^{-o(1)}
\qquad\text{and}\qquad
d(\eps) = o(\ln(1/\eps)),
\]
see \citet[p.~513, Lem.~1]{PW11}, we obtain
\[
\err(A,K) = \eps^{1-o(1)} 
\qquad\text{and}\qquad
\cost(A) \le \eps^{-1/\kappa - o(1)}
\]
using \eqref{g82}.

This shows that if we are able to find sequences
of algorithms on $H(k_1)$ that satisfy the error estimate
\eqref{est:pre_Smolyak_bound} 
for $\kappa$ arbitrarily close to $\min(\dec(k_1), (\rho-1)/2)$, 
we obtain a sequence of MDMs that achieves on $H(K)$ a convergence 
rate arbitrarily close to $\dec(K)$. We know from Section~\ref{s4.1} 
that such algorithms exist in the univariate setting; for integration
constructive results have been derived in \citet{DungNguyen22}.

\appendix

\section{Countable Tensor Products of RKHSs}\label{a1}

In this section we consider Hilbert spaces $H_j \neq \{0\}$
with $j \in \N$ over the same scalar field $\K \in \{\R,\C\}$.
At first, we sketch the construction of the incomplete tensor product
\[
H^{(\bu)} := \bigotimes_{j \in \N} H_j^{(u_j)}
\]
based on any sequence $\bu =(u_j)_{j \in \N}$ of unit vectors 
$u_j \in H_j$. See \citet{Neu39} for a thorough study and 
\citet[App.~A]{GHHR2021} for a more detailed sketch.

In the sequel we use the notation $\bof = (f_j)_{j \in \N}$
with $f_j \in H_j$ for elements $\bof \in \times_{j \in \N} H_j$. 
Let $C^{(\bu)}$ denote the set of all sequences $\bof$ such that 
\[
\sum_{j \in \N} \left| \|f_j\|_{H_j}-1 \right| < \infty
\]
and
\[
\sum_{j \in \N} \left| \scp{u_j}{f_j}_{H_j}-1 \right| < \infty.
\]
It turns out that
\[
\KK^{(\bu)}(\bog,\bof) :=
\prod_{j \in \N} \scp{f_j}{g_j}_{H_j}
\]
is well-defined for all $\bof,\bog \in C^{(\bu)}$ and 
\[
\KK^{(\bu)} \colon C^{(\bu)} \times C^{(\bu)} \to \K
\]
is a reproducing kernel. 
The incomplete tensor product $H^{(\bu)}$ is the RKHS 
$H(\KK^{(\bu)})$, up to an extension of the 
functions $f \in H(\KK^{(\bu)})$ by zero
onto a larger domain $C^{(\bu)} \subsetneq C \subsetneq \times_{j
\in \N} H_j$.

For any $\bof \in C^{(\bu)}$ the function 
$\bigotimes_{j \in \N} f_j := \KK^{(\bu)}(\cdot,\bof)$ 
is called an elementary tensor. Obviously,
$\| \bigotimes_{j \in \N} f_j \|_{H^{(\bu)}} = 
\prod_{j \in \N} \|f_j\|_{H_j}$, and
the span of the elementary tensors is dense in $H^{(\bu)}$.

\begin{rem}\label{ra1} 
Consider any choice of orthonormal bases $(e_{\nu,j})_{\nu \in N_j}$ 
in each of the spaces $H_j$. Assuming $0 \in N_j$ for
notational convenience, we require that
\[
e_{0,j}=u_j
\]
for every $j \in \N$. Let $\bN$ denote the set of all 
sequences $\bn := (\nu_j)_{j \in \N}$ with $\nu_j \in N_j$ for every 
$j \in \N$ and with $\{j \in \N \colon \nu_j \neq 0\}$ being finite.
Then the elementary tensors 
\begin{equation}\label{g45}
e_\bn := \bigotimes_{j \in \N} e_{\nu_j,j}
\end{equation}
with $\bn \in \bN$ form an orthonormal basis of $H^{(\bu)}$.
\end{rem}

Subsequently, we assume that $H_j := H(k_j)$ with reproducing kernels 
\[
k_j \colon D_j \times D_j \to \K.
\]
To $\bx \in D := \times_{j \in \N} D_j$ we associate 
\[
\tau (\bx) := (k_j(\cdot,x_j))_{j \in \N} \in \times_{j \in \N} H(k_j), 
\]
and we put
\[
\X^{(\bu)} := \{\bx \in D \colon \tau(\bx) \in C^{(\bu)} \}.
\]
If $\X^{(\bu)} \neq \emptyset$ then
\[
K^{(\bu)}(\bx,\by) := 
\KK^{(\bu)}(\tau(\bx),\tau(\by)) = 
\prod_{j \in \N} k_j(x_j,y_j)
\]
with $\bx,\by \in \X^{(\bu)}$ yields a reproducing kernel 
\[
K^{(\bu)} \colon \X^{(\bu)} \times \X^{(\bu)} \to \K
\]
of tensor product form. 

In the sense of the following result the incomplete tensor product
$H^{(\bu)}$ is the RKHS with tensor product kernel $K^{(\bu)}$.
This result was first proven in \citet[Thm.~4.10]{Rue20}, see also  \citet[Thm.~A.6]{GHHR2021} for a more succinct version of the proof.

\begin{theo}\label{ta1}
If $\X^{(\bu)} \neq \emptyset$ then
\[
\Phi \colon H^{(\bu)} \to \K^{\X^{(\bu)}},
\]
given by
\[
\phantom{\qquad\quad \bx \in \X^{(\bu)},}
\Phi g (\bx) := g(\tau(\bx)),
\qquad\quad \bx \in \X^{(\bu)},
\]
is an isometric isomorphism between $H^{(\bu)}$ and  $H(K^{(\bu)})$.
In particular, for $\bof \in C^{(\bu)}$ and $\bx \in \X^{(\bu)}$
the product $\prod_{j \in \N} f_j(x_j)$ converges and
\[
\Bigg(\Phi \bigotimes_{j \in \N} f_j \Bigg) (\bx) = 
\prod_{j \in \N} f_j(x_j).
\]
\end{theo}

In addition to the reproducing kernels $k_j$ and the unit vectors
$u_j \in H(k_j)$ we consider reproducing kernels 
$\ell_j \colon D_j \times D_j \to \K$ such that 
$H(k_j) \subseteq H(\ell_j)$ and $\|u_j\|_{H(\ell_j)} = 1$
for every $j \in \N$. Moreover, we let 
\[
T_j \colon H(k_j) \hookrightarrow H(\ell_j)
\]
denote the corresponding identical embedding.
Analogously to $\X^{(\bu)}$ and $K^{(\bu)}$ we consider
the set $\Y^{(\bu)}$ of all $\bx \in D$ such that
\[
\sum_{j \in \N} |\ell_j(x_j,x_j)-1| < \infty
\quad \text{and} \quad
\sum_{j \in \N} |u_j(x_j)-1| < \infty
\]
and the tensor product kernel 
\[
\phantom{\qquad\quad \bx, \by \in \Y^{(\bu)},}
L^{(\bu)} (\bx,\by) := \prod_{j \in \N} \ell_j(x_j,y_j),
\qquad\quad \bx, \by \in \Y^{(\bu)},
\]
assuming that $\Y^{(\bu)} \neq \emptyset$.

\begin{theo}\label{ta2}
If $\emptyset \neq \Y^{(\bu)} \subseteq \X^{(\bu)}$
and $\prod_{j \in \N} \|T_j\| < \infty$ then

\[
\{ f|_{\Y^{(\bu)}} \colon f \in H(K^{(\bu)}) \}
\subseteq H(L^{(\bu)}),
\]
and the operator norm of the restriction
$T \colon H(K^{(\bu)}) \to H(L^{(\bu)})$, 
$f \mapsto f|_{\Y^{(\bu)}}$ is given by
\[
\| T \| = \prod_{j\in \N} \| T_j \|. 
\]
\end{theo}

\begin{proof}
The statement of the theorem is obtained via tensorization
and application of Theorem~\ref{ta1}. More precisely, let 
\[
H^{(\bu)} := \bigotimes_{j \in \N} (H(k_j))^{(u_j)}, \qquad
G^{(\bu)} := \bigotimes_{j \in \N} (H(\ell_j))^{(u_j)},
\]
and let $T^{\otimes} \colon H^{(\bu)} \to G^{(\bu)}$ denote the tensor 
product of the identical embeddings $T_j$.
Note that
\begin{equation}\label{eq:a1}
\| T^{\otimes} \| = \prod_{j\in\N} \| T_j \|,
\end{equation}
which can easily be inferred from the corresponding statement 
for finite tensor products, see, e.g., 
\citet[Prop.~4.150]{Hac12}.
Moreover, let
$\Phi \colon H^{(\bu)} \to H(K^{(\bu)})$ and
$\Upsilon \colon G^{(\bu)} \to H(L^{(\bu)})$ 
denote the isometric isomorphisms according to Theorem~\ref{ta1}.
It suffices to show that
\[
\Upsilon \circ T^{\otimes} \circ \Phi^{-1} f = f|_{\Y^{(\bu)}}
\]
for every $f \in H(K^{(\bu)})$.

Consider any orthonormal bases $(e_\bn)_{\bn \in \bN}$ according to
Remark~\ref{ra1}. Let 
\[
f := \sum_{\bn \in \bN} c_\bn \cdot \Phi e_\bn \in H(K^{(\bu)})
\]
with $c_\bn \in \R$ such that $\sum_{\bn \in \bN} |c_\bn|^2 < \infty$.
We obtain
\[
\Upsilon \circ T^{\otimes} \circ \Phi^{-1} f (\bx) = 
\sum_{\bn \in \bN} c_\bn \cdot \Upsilon \circ T^{\otimes} e_\bn (\bx)
\]
for every $\bx \in \Y^{(\bu)}$. Furthermore,
\[
T^{\otimes} e_\bn = 
\bigotimes_{j \in \N} T_j e_{\nu_j,j}
\in G^{(\bu)},
\]
and therefore
\[
\Upsilon \circ T^{\otimes} e_\bn (\bx) =
\prod_{j \in \N} e_{\nu_j,j} (x_j)
= \Phi e_\bn (\bx) 
\]
for every $\bx \in \Y^{(\bu)}$. It follows that
\[
\Upsilon \circ T^{\otimes} \circ \Phi^{-1} f (\bx) = f(\bx)
\]
for every $\bx \in \Y^{(\bu)}$, as claimed. 

The statement on the operator norm of $T$ follows from
\eqref{eq:a1} as well as the fact that $\Phi$ and $\Upsilon$ 
are isometries.
\end{proof}

\section{Monotonicity of the Univariate Hermite Kernel}\label{a2}

In this section we consider a non-decreasing sequence 
$\ba := (\alpha_{\nu})_{\nu\in \N}$ 
of Fourier weights satisfying \eqref{g11}
and the corresponding univariate Hermite kernel, see \eqref{g15}.

\begin{lemma}\label{l-a1}
For $x,y\in\R$ with $|x| \leq |y|$ we have
\[
k(x,x)\leq k(y,y).
\]
\end{lemma}

\begin{proof}
We observe that $k(x,x) = k(-x,-x)$
hold for all $x\in\R$, since each of the Hermite polynomials
is either an even or an odd function. Thus, it suffices to show that 
$x\mapsto k(x,x)$ is non-decreasing on $\R^+$.
To this end, we use two recurrence relations of Hermite polynomials,
namely
\[
h'_\nu(x)=\nu^{1/2} \cdot h_{\nu-1}(x) \quad\text{and}\quad
h_\nu(x) = \nu^{-1/2} \cdot (xh_{\nu-1}(x)-h'_{\nu-1}(x))
\]
for $\nu\in\N$ and $x\in\R$, see \citet[Eqn.~(5.5.10)]{S75}.
This implies 
\[
(h_\nu^2)'(x) = 2xh_{\nu-1}^2(x)-(h_{\nu-1}^2)'(x)
\]
and thus, inductively,
\begin{equation}\label{eq:kernel_derivative}
(h_\nu^2)'(x) =2x \cdot 
\sum_{\kappa=1}^{\nu}(-1)^{\nu-\kappa} \cdot h_{\kappa-1}^2(x).
\end{equation}
For $n\in\N$ and $x \in \R_+$ we consider the partial sum
\[
f_n(x):=\sum_{\nu=0}^{n}\alpha_{\nu}^{-1} \cdot h_\nu^2(x).
\]
By \eqref{eq:kernel_derivative} we have
\begin{align*}
f'_n(x)
&= 
2x \cdot \sum_{\nu=1}^{n} 
\alpha_{\nu}^{-1} \sum_{\kappa=1}^{\nu} 
(-1)^{\nu-\kappa} h_{\kappa-1}^2(x)\\
&= 2x \cdot
\sum_{\kappa=1}^{n}h_{\kappa-1}^2(x)\sum_{\nu=0}^{n-\kappa}
(-1)^{\nu}\alpha_{\nu+\kappa}^{-1}.
\end{align*}
Since $\ba$ is non-decreasing, the inner sum on the right-hand side
is always non-negative, and therefore
$f_n$ is non-decreasing on $\R^+$. 
Finally, $f_n(x)$ converges to $k (x,x)$ for every $x\in\R^+$.
\end{proof}

\subsection*{Acknowledgment}

Part of this work was done during the Research Meeting 22073 
``Complexity of Infinite-Dimensional Problems'' at Schlo\ss\ Dagstuhl.
The authors would like to thank the staff of Schlo\ss\ Dagstuhl 
for their hospitality.

We are grateful to Yuya Suzuki for pointing out the paper
\citet{DungNguyen22} to us.

{\small

\bibliographystyle{abbrvnat}
\bibliography{references}

}

\end{document}